\documentclass[12pt,a4paper]{article}
\usepackage[utf8]{inputenc}

\usepackage{amsmath,amsthm,amssymb,geometry,todonotes,mathtools,overpic}
\usepackage{enumerate}
\usepackage{color}
\usepackage{hyperref}
\hypersetup{
	colorlinks,
	linkcolor={red!50!black},
	citecolor={blue!50!black},
	urlcolor={blue!80!black}
}

\definecolor{lightblue}{RGB}{0,153,255}
\definecolor{darkblue}{RGB}{0,51,204}
\definecolor{purple}{RGB}{102,0,204}

\DeclareMathOperator{\aff}{aff}
\DeclareMathOperator{\cl}{cl}
\DeclareMathOperator{\cone}{cone}
\DeclareMathOperator{\conv}{conv}
\DeclareMathOperator{\dist}{dist}

\DeclareMathOperator{\lspan}{span}
\DeclareMathOperator{\relint}{relint}
\DeclareMathOperator{\R}{\mathbb{R}}
\newcommand{\norm}[1]{\left\|#1\right\|}
\newcommand{\stdCone}{ {\mathcal{K}}}
\newcommand{\inProd}[2]{\langle #1 , #2 \rangle }
\newcommand{\face}{\mathrel{\unlhd}}

\newtheorem{lemma}{Lemma}[section]

\newtheorem{theorem}[lemma]{Theorem}

\newtheorem{proposition}[lemma]{Proposition}
\newtheorem{remark}[lemma]{Remark}

\theoremstyle{definition}

\numberwithin{equation}{section}

\title{Inner approximations of convex sets and intersections of projectionally exposed cones}
\author{Bruno F.\ Louren\c{c}o\thanks{Department of Fundamental Statistical Mathematics, Institute of Statistical Mathematics, Japan. Email:~\texttt{bruno@ism.ac.jp}} \and Vera Roshchina\thanks{School of Mathematics and Statistics, UNSW Sydney, Australia. Email:~\texttt{v.roshchina@unsw.edu.au}} \and James Saunderson\thanks{Department of Electrical and Computer Systems Engineering, Monash University, Australia. Email:~\texttt{james.saunderson@monash.edu}}}

\begin{document}
	
	\maketitle
	
	\begin{abstract}

		A convex cone is said to be projectionally exposed (p-exposed) if every face arises as a projection of the original cone. It is known that, in dimension at most four, the intersection of two p-exposed cones is again p-exposed. In this paper we construct two p-exposed cones in dimension $5$ whose intersection is not p-exposed. This construction also leads to the first example of an amenable cone that is not projectionally exposed, showing that these properties, which coincide in dimension at most $4$, are distinct in dimension $5$. 
		In order to achieve these goals, we develop a new technique for constructing arbitrarily tight inner convex approximations of compact convex sets with desired facial structure. 
		These inner approximations have the property that all proper faces are extreme points, with the exception of a specific exposed face of the original set.


%%		A convex cone is said to be projectionally exposed (p-exposed) if every face arises as a projection of the original cone. Previously, it was known that amenable cones in dimension $4$ or less must be p-exposed. In this paper, we construct for the first time a dimension $5$ amenable cone that is not p-exposed and, while doing so, we show that projectional exposedness is not preserved under intersections.
%%		In order to achieve these goals, we develop a new technique for constructing arbitrarily tight inner convex approximations of compact convex sets with desired facial structure (namely, such that the proper faces of these inner approximations are extreme points, with the exception of a specific exposed face of the original set). 
		
%		This construction is used to prove that there exist projectionally exposed cones in $\R^5$ such that their intersection is not projectionally exposed, disproving the conjecture that the intersection of projectionally exposed cones is projectionally exposed, and demonstrating that amenable cones are not necessarily projectionally exposed in Euclidean spaces of dimension 5 and higher.
	\end{abstract}
	
	\section{Introduction}\label{sec:int}
	Given a closed convex cone $\stdCone \subseteq \R^n$, one approach to analyzing its geometric properties is through examining its \emph{faces}. 
	In this context, the study of different notions of \emph{facial exposedness} is a classical subject.
	In particular, a face $F$ of $\stdCone$ is said to be \emph{facially exposed} if there exists some supporting hyperplane $H$ of $\stdCone$ such that $F = \stdCone \cap H$ holds.
	
	Facial exposedness, although useful, is a relatively weak property. 
	Motivated by specific applications, several researchers have proposed stronger facial exposed properties. Here we recall a few, deferring precise definitions to Section~\ref{sec:exp_def}. 
	In connection to their facial reduction algorithm, Borwein and Wolkowicz introduced the notion of \emph{projectional exposedness} in \cite{BW81}. 
	Also, in the same paper, they described a condition that would later be known as \emph{niceness}, see \cite[Remark~6.1]{BW81} and \cite[Definition~1.1]{PatakiLI}.
	Niceness turned out to be a very fruitful property with connections to extended conic duals \cite{Pa13_2} and to the problem of representing certain convex sets as lifts of convex cones \cite[Corolary~1]{GPR13}. Niceness can also be used to give necessary and sufficient conditions for the image of a closed convex cone  by a linear map to be closed, see \cite[Theorem~1.1]{PatakiLI}.
	Later, motivated by error bounds in conic programming, the notion of \emph{amenable cones} was proposed in \cite{L19}.
	
	Given the diversity of exposedness properties, it is of interest to understand how these properties are related.
	In a nutshell, we have the following inclusions:
	\begin{equation}\label{eq:prop_rel}
		\text{ facially exposed cones} \supsetneq \text{ nice cones} \supsetneq \text{amenable  cones} \underset{?}{\supsetneq} \text{projectionally exposed cones}.
	\end{equation}
	We now discuss the inclusions in \eqref{eq:prop_rel}.
	Pataki showed in \cite[Theorem~3]{pataki} that all nice cones are facially exposed and conjectured that the converse holds. 
	A counter-example consisting of a cone in $\R^4$ that is facially exposed but not nice was later constructed in \cite{Vera}. 
	Every projectionally exposed cone is amenable and amenable cones are nice, see \cite[Propositions~9 and 13]{L19}. 
	An example of a cone in $\R^4$ that is nice but not amenable was given in \cite{LRS20}. A natural question is \emph{whether there exists a cone that is amenable but not projectionally exposed}.  
%	\todo[inline]{B: Cite \cite{RT} somewhere appropriate.}
	
	The fact that the aforementioned (counter-)examples  are in $\R^4$ is not a coincidence. 
	In fact, projectional exposedness coincides with facial exposedness in $\R^3$, i.e., we can only distinguish those  exposedness properties in dimension at least four, 
	see \cite[Theorem~3.2]{PL88} or \cite[Theorem~4.6]{ST90}.
	% Showing that one facial exposedness property implies another requires a \emph{proof} but showing that two properties are not the same requires a \emph{counter-example}. 
	%Coming up with a specific counter-example of a cone in $\R^4$ satisfying certain properties but not others is challenging but there is a manageable aspect to it: a cone in $\R^4$, under appropriate assumptions, has three-dimensional slices that encapsulate its facial structure. 
	%In a sense, one only has to construct an appropriate convex set in $\R^3$, a task that is more amenable to our usual geometric sensibilities.  
	Curiously, in \cite[Corollary~6.4]{LRS20}, we showed that amenability and projectional exposedness coincide in dimension four. Therefore, a potential example of an amenable but not projectionally exposed cone must have dimension at least five. 
	%This creates some challenges since slices of convex cones in $\R^5$ would correspond to convex sets of dimension $4$. That is also why the last inclusion in \eqref{eq:prop_rel} has a question mark. Previous work has not identified any examples of amenable but not projectionally exposed cones.
	
	Insight about a property of convex cones can be gained by understanding which convexity-preserving operations also preserve that property. It turns out that the intersection of two amenable cones is an amenable cone~\cite{LRS20}. (The same result holds for nice cones and for facially exposed cones.) It is then natural to ask \emph{whether the intersection of two projectionally exposed cones is always projectionally exposed}. 
	Again, any counterexamples must have dimension at least five.

%Related to that, it is known that, with the possible exception of projectional exposedness, the other properties in \eqref{eq:prop_rel} are preserved under intersection. 
%	And, again, any examples of projectionally exposed cones whose intersection is not projectionally exposed cone must have dimension at least 5.
	
	In this paper, we address these two questions and show that in $\R^5$:
	\begin{enumerate}[$(a)$]
		\item there exist two projectionally exposed cones whose intersection is not projectionally exposed;
		\item there exists an amenable but not projectionally exposed cone.
	\end{enumerate}
	Since amenability is preserved under intersections \cite[Proposition~3.4]{LRS20} and projectionally exposed cones are amenable, goal (b) can be accomplished by intersecting the cones obtained in (a).
	
	The basic strategy we adopt is to start by constructing a `local' counterexample. We construct  
	two convex cones $\stdCone_i$ (for $i=1,2$), and focus on a single face $F_i$ (for $i=1,2$) of
	each that is projectionally exposed. We show that the intersection of these faces is not 
	a projectionally exposed face of $\stdCone_1\cap \stdCone_2$. To turn this into a true counterexample, 
	we need to ensure that the rest of the faces of $\stdCone_1$ and $\stdCone_2$ are all projectionally 
	exposed. However, our method of constructing the cones makes it difficult to even determine the explicit
	facial structure, let alone understand their facial exposedness properties.

%	A challenge in our analysis  is that we want to identify one specific face of the cone that is ``bad'' (i.e., fails projectional exposedness) but the cone should still well-behaved elsewhere since it must be amenable. 
	
	With this in mind, the other main goal of this paper is to develop a set of tools for modifying a convex set $C$ in a controlled way so that the local behavior of $C$ with respect to some face $F$ is preserved but the facial structure of the modified set, apart from $F$, is well-behaved and easy to understand. We then apply these 
	tools to modify the cones $\stdCone_1$ and $\stdCone_2$ in such a way that all of their faces are projectionally exposed, and the `bad' faces $F_1$ and $F_2$ are preserved.
	
	There are other ways to simplify the structure of convex sets in high-dimensional examples, where it may be difficult to analyse the overall facial structure. For instance, in \cite[Example~6.7]{ST90} this is done by ensuring that the convex set is locally polyhedral everywhere apart from the neighbourhood of one particular point that exemplifies the behaviour in question.  
	
%	apart from $F$.
%	This will make it easier to construct and verify examples where only a single face of $C$ is ``bad'' in some sense.
		
%		\todo[inline]{B: }

	We prove several technical results related to the construction of inner approximations of convex sets. %and apply these new results to the proof intersections of projectionally exposed cones may not be projectionally exposed in the Euclidean space of dimension 5. 
	We can construct as tight an inner approximation of a compact convex set as we wish, preserving one specific face of the original set, and ensuring that apart from this one face, the inner approximation has a very simple facial structure. More precisely, we have the following result that summarises the approach. 
	
	\begin{theorem}\label{thm:amalgamation} For any compact convex set $C\subset\R^n$, a proper nonempty exposed face $F$ of $C$ and any continuous function  $\varphi: C\to \R_+$ such that $\varphi(x)=0$ if and only if $x\in F$, there exists a compact convex set ${D}\subseteq C$ such that 
		\begin{enumerate}
			\item 		$\dist(x,{D}) \leq \varphi(x) \; \forall\, x\in C$,
			\item $F\subseteq \partial C\cap D$, and $F$ is a face of $D$, 
			\item every face of $D$ is either a subface of $F$ or an extreme point. 
		\end{enumerate}            
	Here by $\partial C$ we denote the relative boundary of $C$.
	\end{theorem}

	We prove Theorem~\ref{thm:amalgamation} at the end of Section~\ref{sec:sandwich}.
	
	\begin{remark} A natural choice of the function $\varphi$ is some power of the distance to a hyperplane exposing a face $F$ of $C$, that is, 
		\[
		\varphi(x) = \min_{u\in H} \|x-u\|^\alpha.
		\]
		For instance, in the proof of Proposition~\ref{prop:approxwtg} we use $\alpha = 2$ to ensure that the inner approximation shares the tangent cones with the original set at all points of $F$.
	\end{remark}
	
%	The second major contribution of the paper is the proof that projectional exposure is not closed under intersections.
	
%	Recall that a face $F\unlhd K$ is  \emph{projectionally exposed} if there exists an 
%	idempotent linear map $P: \R^n\to \R^n$ (i.e., a linear projection that is not necessarily orthogonal) such that
%	\[
%	P(K) = F.
%	\]
%	A cone $K$ is projectionally exposed if  every face of this cone is projectionally exposed.
	
	With the aid of the tools developed in Section~\ref{sec:tools} we suitably modify a certain basic construction to prove the following.
	
	\begin{theorem}\label{thm:main} There exist two closed convex cones $\tilde \stdCone_1,\tilde \stdCone_2\subset \R^5$ such that both $\tilde \stdCone_1$ and $\tilde \stdCone_2$ are projectionally exposed, but their intersection $\tilde \stdCone_1\cap \tilde\stdCone_2$ is not.
	\end{theorem}
In particular, $\tilde \stdCone_1\cap \tilde\stdCone_2$ in Theorem~\ref{thm:main} is a closed convex cone that is amenable but not projectionally exposed.
%	
%	Recall \cite{L19} that a closed convex cone $K$ is said to be \emph{amenable}, if for every face $F\unlhd K$ there exists a constant $\kappa  > 0$ (possibly depending on $F$) such that
%	\begin{equation}\label{eq:def_am_cone}
%	\dist(x, F) \leq \kappa \dist(x, K), \quad \forall x \in \lspan F.
%	\end{equation}
%
%	\begin{remark} 	
%	It  was shown in \cite{L19} that projectionally exposed cones are amenable, and  in \cite{LRS20} it was shown that amenable cones are projectionally exposed in dimensions up to and including four. It is also known that amenable cones are closed under intersection. That is why Theorem~\ref{thm:main} can't be true in $\R^n$ if $n\leq 4$. 
%	\end{remark}	
%
%	
%	\begin{remark}
%		We know that projectionally exposed cones are amenable, and that intersections of amenable cones are amenable. If $\tilde K_1$ and $\tilde K_2$ are projectionally exposed, then they must both be amenable. Their intersection must be amenable, but it is not projectionally exposed. Therefore amenable cones are not necessarily projectionally exposed in dimensions 5 and higher. 
%	\end{remark}

	Our paper is organised as follows. We begin with laying out the notation, definitions and context related to different notions of facial exposedness in Section~\ref{sec:prel}. Then in Section~\ref{sec:tools} %~\ref{sec:facefix} and~\ref{sec:sandwich} %
	we prove the two main technical results of this paper. The first one, Lemma~\ref{lem:DclosebdC}, ensures that it is possible to construct an arbitrarily tight compact convex inner approximation to a closed convex set $C$ such that the boundary of this approximation only intersects $C$ at a specific face $F$. The proof of Lemma~\ref{lem:DclosebdC} is fairly straightforward and is based on constructing a compact set that satisfies tightness requirement and then taking its convex hull. The second technical result  given in Lemma~\ref{lem:refinement} allows us to sandwich a compact convex set with convenient facial structure between two other compact convex sets.

	In Section~\ref{sec:counterexample} we prove Theorem~\ref{thm:main}. 
	To simplify the explanation, in Section~\ref{sec:construction} we first construct two closed convex cones $\stdCone_1$ and $\stdCone_2$ with specific projectionally exposed faces $F_1$ and $F_2$, such that the intersection of these faces $F= F_1\cap F_2$ is a face of $\stdCone_1\cap \stdCone_2$, that is not projectionally exposed. 
	We then give the full details of the final construction in the proof of Theorem~\ref{thm:main} in Section~\ref{sec:proof}.
	 We obtain it using the inner approximation toolbox designed earlier, so that it is straightforward to  prove that the remaining faces of the modified cones $\tilde \stdCone_1$ and $\tilde \stdCone_2$ are projectionally exposed. 
	 This is done by ensuring that all faces of these modified cones except for $F_1$ and $F_2$ are extreme rays, and hence are automatically projectionally exposed. 

	We conclude this paper in Section~\ref{sec:conc} in which we also discuss some open problems. 
	
	\section{Preliminaries and notation}\label{sec:prel}
	We suppose that $\R^n$ is equipped with an inner product $\inProd{\cdot}{\cdot}$ and a corresponding induced norm $\norm{\cdot}$. 
	For a set $S \subseteq \R^n$ we will write $S^\perp$ for the set of elements orthogonal to $S$ with respect to $\inProd{\cdot}{\cdot}$. We denote the convex hull and the conic convex hull of $S$ by 
	$\conv S$ and $\cone S$, respectively.
	If $x,y \in \R^n$, we will use $[x,y]$ to indicate the line segment connecting $x,y$, i.e.,
	\[
	[x,y] \coloneqq \{ \alpha x + (1-\alpha)y \mid \alpha \in [0,1] \}.
	\]
	We will also use $(x,y]$, $[x,y)$ and $(x,y)$ to denote $[x,y]\setminus \{x\} $ $[x,y] \setminus \{y\}$ and $[x,y] \setminus \{x,y\}$, respectively.

	Let $C \subseteq \R^n$ be a convex set. 
	We denote its closure, relative interior and relative boundary by $\cl C$, $\relint C$ and $\partial C$, respectively. 
	The span and the affine hull of $C$ will be denoted by 
	$\lspan C$ and $\aff C$, respectively.
	The distance function with respect to $C$ is given by 
	\[
	\dist(x, C) \coloneqq \inf _{y \in C} \norm{x-y}.
	\]
	We also note the following fact about relative interiors that will be used several times.
	%	We first recall the fundamental property of relative interiors.
	\begin{equation}\label{eq:risegm}
	x \in \relint C, y \in \cl C \,\,\Rightarrow\,\,[x,y)\subseteq \relint C,
	\end{equation}
	see \cite[Theorem~6.1]{Ro97}.
%	\begin{theorem}[{\cite[Theorem~6.1]{Ro97}}]\label{thm:risegm} Let $C$ be a convex set in $\R^n$. Let $x\in \relint C$ and $y\in \cl C$. Then $[x,y)\subseteq \relint C$.
%		
%	\end{theorem}
	
	Let $x \in C$.
	A half-space $H \coloneqq \{y \in \R^n \mid 
	\inProd{y-w}{p} \geq 0\}$ is said to \emph{support $C$ at $x$}, if $C \subseteq H$ and $x$ is in the underlying hyperplane, i.e., $\inProd{x-w}{p} = 0$.
	Similarly, a hyperplane is said to support $C$ at $x$, if $x$ belongs to it  and $C$ is entirely contained in one of the half-spaces defined by it. 
	In some proofs it will be more convenient to use half-spaces and in others hyperplanes. %, so we will use both.

	If $F \subseteq C$ is a face of $C$, we will express this with the notation $F \face C $. A face $F \face C$ is said to be \emph{proper} if $F \neq C$. 
	By convention, we assume that a face must be nonempty.
	
	For a convex cone $\stdCone \subseteq \R^n$, we denote its dual by $\stdCone^*:=\{y\in \R^n \mid \inProd{y}{x} \geq 0,\;\forall x\in \stdCone\}$.
%\todo[inline]{J: add definition of $\stdCone^*$ somewhere?}

	\subsection{Notions of exposedness}\label{sec:exp_def}
	Let $\stdCone \subseteq \R^n$ be a closed convex cone.
	A face $F \face \stdCone$ is said to be \emph{exposed} if there exists $y \in \stdCone^*$ such that 
	\[
	F = \stdCone \cap \{y\}^\perp
	\]
	holds. 	 A convex cone $\stdCone$ is said to be \emph{facially exposed} if all its faces are facially exposed.
	
	A face $F\face \stdCone$ is said to be \emph{amenable} if there exists a constant $\kappa  > 0$  such that
	\begin{equation}\label{eq:def_am_cone}
		\dist(x, F) \leq \kappa \dist(x, \stdCone), \quad \forall x \in \lspan F
	\end{equation}
	holds,	see \cite{L19,LRS20} for more details. The error bound condition in \eqref{eq:def_am_cone} is motivated by the fact that for $F \face \stdCone$ we always have
	\begin{equation}\label{eq:f_span}
	F = \stdCone \cap (\lspan F),
	\end{equation}
	so \eqref{eq:def_am_cone} is a strengthening of \eqref{eq:f_span} that gives an upper bound on 
	$\dist(x, F)$ of order $O(\dist(x, \stdCone))$ when $x$ is in $\lspan F$ but not necessarily in $\stdCone$. 
	If all faces of $\stdCone$ are amenable, 
	then $\stdCone$ is said to be an \emph{amenable cone}. 
	We note that in this case, the constant $\kappa$ in  \eqref{eq:def_am_cone} may depend on $F$.
	Examples of amenable cones include the hyperbolicity cones \cite{LRS24}, which is a broad family of cones containing the symmetric cones and the cones of positive semidefinite matrices. 
	In particular, 	spectrahedral cones (which arise by slicing the cone of positive semidefinite matrices with a subspace) are amenable too, see \cite[Corollary~3.5]{LRS20}.
	The intersection of two amenable cones is amenable \cite[Proposition~3.4]{LRS20}.
	
	A face $F\face \stdCone$ is said to be \emph{projectionally exposed} (p-exposed) if there exists an 
	idempotent linear map $P: \R^n\to \R^n$ (i.e., a linear projection that is not necessarily orthogonal) such that
	\[
	P(\stdCone) = F
	\]
	holds, see \cite{BW81,ST90} for more details.
	If every face of $\stdCone$ is projectionally exposed then $\stdCone$ is said to be projectionally exposed.
	Examples of projectionally exposed cones include  symmetric cones \cite[Proposition~33]{L19}, generalized power cones \cite[Proposition~3.4]{LLLP24} and all polyhedral cones \cite[Therem~2.4]{BLP87} \cite[Corollary~3.4]{ST90}. 	 
	However, it is not currently known whether hyperbolicity cones  or spectrahedral cones are projectionally exposed in general.
	In contrast to amenability, we will show in this paper that projectional exposedness is not preserved under intersections.
	
	%For the sake of completeness, we mention that 
	A face $F \face \stdCone$ is said to be \emph{nice} if 
	$F^* = \stdCone^* + F^\perp$ holds. 
	Again, this is motivated by \eqref{eq:f_span} which implies 
	that $F^* = \cl(\stdCone^* + F^\perp)$ always holds in general. 
	In this way, a nice face is one for which the closure operator can be removed from the expression for $F^*$. If all faces of $\stdCone$ are nice, then $\stdCone$ is said to be a nice cone or \emph{facially dual complete} \cite{RT}. 
	As mentioned in Section~\ref{sec:int}, there are several applications of niceness \cite{PatakiLI,GPR13,Pa13_2}.
	
	Niceness is a dual condition and it can be hard to verify directly.
	Fortunately, a projectionally exposed face must be amenable and an amenable face must be nice\footnote{This follows from the proofs of \cite[Propositions~9 and 13]{L19} which show that projectionally exposed cones are amenable and amenable cones are nice, respectively. However, both propositions are proven face-by-face.}. %\todo{B: Add proofs to that? J: Maybe just references?}
%	\todo[inline]{B: Mention it is proved in \cite{L19} for cones, but the proof is face by face and etc.}
	In this way, both properties can be seen as primal sufficient conditions for niceness and may be simpler to verify in some cases, see also \cite[Theorem~2.7]{RT}.
	For example, as in \cite[Proposition~3.4]{LLLP24}, it is somewhat straightforward to check that power cones are projectionally exposed. 
	For hyperbolicity cones, as far as we know, there was no independent proof that hyperbolicity cones are nice. 
	We now know they are nice because they are amenable \cite{LRS24}, but before that, the best result available for hyperbolicity cones was that they are facially exposed \cite[Theorem~23]{Re06}.
%	\todo[inline]{B:Here it might be a good to place to cite \cite{RT} in the context of sufficient conditions for niceness}
	
	Finally, we note that a nice face is not necessarily facially exposed and this does not contradict \eqref{eq:prop_rel}. 
	Although a nice face may not be facially exposed, a nice cone (i.e., \emph{all} its faces are nice) must indeed be facially exposed by \cite[Theorem~3]{pataki}.
	
	\section{Approximating convex sets}\label{sec:tools}
	In this section, we develop our machinery to obtain inner approximation of convex sets that fix a given face. 
	
	\subsection{Face-fixing inner approximations of convex sets}\label{sec:facefix}
	We start with a result that tells us how to find an inner approximation of a compact convex set $C$ that fixes a given face $F \face C$. This inner approximation will depend on a given function $\varphi$ that is $0$ if and only if $x \in F$.
	
	\begin{lemma}\label{lem:DclosebdC}
		Suppose that $C \subset \R^n$ is a compact convex set, and suppose that $F$ is a proper face of $C$. If a function $\varphi: C\to \R_+$ is continuous, and $\varphi(x)=0$ if and only if $x\in F$, then there exists a compact convex set ${C'}$ such that 
		\begin{equation}\label{eq:lem:DclosebdC}
		\dist(x,{C'}) \leq \varphi(x), \; \forall\, x\in C,  
		\end{equation}
		$\relint {C'}\subseteq \relint C$ and $\partial C\cap {C'} = F$.
	\end{lemma}
%\[
%x \in \relint C, y \in \cl C \Rightarrow [x,y)\subseteq \relint C
%\]

	\begin{proof}%[Proof of Lemma~\ref{lem:DclosebdC}]
		Fix some $c\in \relint C$, so that  $\aff C = L+c$ holds, where $L$ is a linear subspace. There exists a sufficiently small $r>0$ such that   $c+ v \in \relint C$ for all $v\in L$ with $\|v\|\leq r$. For any $x\in \partial C$ let 
		\[
		\lambda(x) \coloneqq \min\{\min_{u\in [x,c]} \varphi(u), \|x-c\|-r\}.
		\]
		\emph{Roughly speaking, this quantity $\lambda(x)$ determines the width of the `gap' that we will have between the boundary of the original set $C$ and the boundary of the newly constructed set $C'$. 
			The first term ensures that we stay within the required margin, and the second term ensures that the new set includes a neighbourhood of $c$ with respect to $\aff C$.}
		
		Throughout this proof, it is useful to keep the following fact in mind. Because of the choice of $r$, we have
		\begin{equation}\label{eq:norm_reminder}
			\norm{x-c} > r, \quad \forall x \in \partial C,
		\end{equation}
		since otherwise $x = c + (x-c)$ would belong to $\relint C$.
		With that, $\lambda(x)$ is nonnegative for any $x\in \partial C$, since $\varphi(x)\geq 0$ by assumption. 
		From the definition of $\lambda$, we conclude that $\lambda(x) =0$ if and only if there is a point on $u\in [x,c]$ such that $\varphi(u)=0$, equivalently $u\in F$. 
		
		Next, we will check that for $x \in \partial C$, we have
		\begin{equation}\label{eq:lambdax}
			\lambda(x)=0 \Leftrightarrow x\in F.
		\end{equation}
		First suppose that $x \in \partial C$ is such that $\lambda(x) = 0$. In view of \eqref{eq:norm_reminder}, we must have $\min_{u\in [x,c]} \varphi(u) = 0$, which means that $u \in F$ for some $u \in [x,c]$.
		Recalling \eqref{eq:risegm} we have $(x,c]\subset\relint C$. 
		Since $F$ is a proper face, we have $F\cap \relint C = \emptyset$, therefore such a $u$ cannot be in 
		$(x,c]$ and we must have $u = x$.
		Conversely, if $x \in F$, then $\varphi(x) = 0$ so $\lambda(x) = 0$.

		We also note that $\lambda$ is continuous: the minimum of a continuous function over a compact set of parameters is continuous\footnote{For instance this follows from Berge's theorem, see \cite[Theorem~17.31]{Hitchhiker}.}, and the minimum of two continuous functions is continuous. 
		
		For every $x\in \partial C$ let 
		\begin{equation}\label{eq:ux}
		u(x) \coloneqq x+ \lambda(x) \frac{c-x}{\|c-x\|} = x\left(1- \frac{\lambda(x)}{\norm{c-x}}\right) +  \frac{\lambda(x)}{\norm{c-x}}c,
		\end{equation}
		We note that since $\lambda(x) \leq \norm{x-c} -r \leq \norm{x-c}$, we have $\frac{\lambda(x) }{\norm{c-x}} \leq 1$, so $u(x) \in C$.
		Furthermore, in view of \eqref{eq:risegm}, the definition of $u(x)$ and \eqref{eq:lambdax}, we have the following implications for $x \in \partial C$
		\begin{equation}\label{eq:uxf}
			u(x) \in \partial C \Leftrightarrow u(x) = x \Leftrightarrow \lambda(x) = 0 \Leftrightarrow x \in F.
		\end{equation}
		
		We define
		\[
		U \coloneqq \bigcup_{x\in \partial C} [u(x),c].
		\]
		Note that 
		\begin{equation}\label{eq:ucx}
			[u(x),c] = \left\{x+ \lambda \frac{c-x}{\|c-x\|}\, |\, \lambda \in [\lambda(x),\|c-x\|]\right\}.
		\end{equation}
		We have $U\subseteq C$ and that $U$ is bounded, since $C$ is assumed to be compact. 
		
		Next, for each $v\in L$ with $0<\|v\|\leq r$ the half-ray $\{c+\beta v \mid \beta \geq 0\}$ intersects $C$ at some relative boundary point $x$. That is,
		\[
		c + \beta v = x \in \partial C,
		\]
		for some $\beta > 0$. However, by the choice of $r$, we must have $\beta > 1$ (otherwise, $x$ would be in $\relint C$). That is,
		$v = (x-c)/\beta$ holds for some $\beta > 1$.
		
		Equivalently, since  $\norm{x-c} > r$ holds by \eqref{eq:norm_reminder} and $\norm{v} \leq r$, letting $\alpha \coloneqq \frac{\norm{x-c}}{\beta} = \norm{v}$ we have
		\begin{equation}\label{eq:v_expression}
			v = \alpha \frac{x-c}{\|x-c\|} 
		\end{equation}
		and $\alpha \in (0,r]$. Then,
%		\[
%		u(x) = x+ \lambda(x) \frac{c-x}{\|c-x\|} =c + \left(1-\frac{\lambda(x)}{\|c-x\|} \right) (x-c),
%		\]
		since $\lambda(x)\leq  \|x-c\|-r$ holds, we have 
		\[
		\norm{x-c} \geq \|x-c\|-\alpha \geq \norm{x-c}-r \geq \lambda(x)
		\]	
%		\[
%		1-\frac{\lambda(x)}{\|c-x\|}\geq  1+\frac{r-\|c-x\|}{\|c-x\|} =\frac{r}{\|c-x\|}\geq \frac{\alpha}{\|c-x\|}, 
%		\]
		and hence, in view of \eqref{eq:ucx} and \eqref{eq:v_expression}, leads to
			\[
			c+v = x+ (\norm{x-c}-\alpha)\frac{c-x}{\|x-c\|} %= x\left(\frac{\alpha}{\norm{c-x}}\right) + \left(1 - \frac{\alpha}{\norm{c-x}} \right)c  %=c + \left(\frac{\alpha}{\norm{c-x}}\right)(x-c) 
			\in [u(x),c]
			\]
			This shows that the ball of radius $r$  around $c$
			intersected with $\aff C$ is in $U$.

		Let us now show that $U$ is closed, and hence compact (it is bounded as a subset of the compact set $C$). Let $(u_k)$ be a sequence in $U$ converging to some $\bar u$. Then there is an accompanying sequence $(x_k)\subset \partial C$ such that 
		\[
		u_k = x_k+ \lambda_k \frac{c-x_k}{\|c-x_k\|},
		\]
		for some $\lambda_k \in [\lambda(x_k), \|c-x_k\|]$, see \eqref{eq:ucx}. 
		
		Without loss of generality $x_k \to \bar x$, and $\lambda_k \to \bar \lambda$ (recall that $\lambda$ is continuous on the compact boundary of $C$, hence its image is compact). Then 
		\[
		\bar u = \lim_{k \to \infty} u_k = \bar x + \bar \lambda  \frac{c-\bar x}{\|c-\bar x\|}.
		\]
		Since $\lambda$ is continuous, and {$\lambda(x_k) \leq \lambda_k $ (see \eqref{eq:ucx}), we have $\lambda(\bar x) \leq \bar \lambda $}. Moreover since $\lambda_k \leq \|x_k-c\|$, we also have $\bar \lambda \leq \|\bar x - c\|$. 
		In view of \eqref{eq:ucx}, we conclude that  $\bar u \in [u(\bar x), c]\subseteq U$.

		Now let ${C'} = \conv U$. The set ${C'}$ is convex, compact and contains a neighbourhood of $c$ (with respect to the affine hull of $C$) in its relative interior. To see that $\partial C\cap {C'} =F$, first observe that for any $x\in F$, we have 
			$u(x) = x \in \partial C$ by \eqref{eq:uxf}.
			In particular, $x\in U\subseteq {C'}$, so \[F\subseteq \partial C\cap {C'}.\]

		Conversely, let $y\in \partial C\cap {C'}$.
		Then, we can express $y$ as a strict convex combination of elements of $U$ so that 
		\[
		y = \sum _{i=1}^m \alpha _i u _i,
		\]
		where the $\alpha_i$'s are in $(0,1)$,  $\alpha_1 + \cdots + \alpha_m = 1$ holds and the $u_i$ are in $U$. By definition, each $u_i$ is in some line segment $[u(x_i),c]$, for some $x_i \in \partial C$. We also have $(u(x_i),c] \subset\relint C$, by \eqref{eq:risegm}.
		Therefore, if any single $u_i$ is in $(u(x_i),c] \subset\relint C$, then $y$ would be in $\relint C$ as well. This cannot happen by assumption, so $u_i = u(x_i)$ holds for every $i$ and, since $u(x_i) \in C$, it must be the case that $u(x_i)$ is in the relative boundary $\partial C$. By \eqref{eq:uxf}, we conclude that $u(x_i) = x_i$ and $x_i \in F$, for every $i$. Therefore, $y \in F$.
		
		%		Conversely, assume the contrary and take $y\in (\partial C\cap {C'})\setminus F$. Then there is some face $E$ of $C$ such that $y\in \relint E$. If $y$ is represented as a strict convex combination of some points in $C$, all these points must belong to the face $E$. Since $y\in {C'}$, it must be represented as a convex combination of points in $U$, and so all of these points are in $E$. But the only points in $U$ that are also on the boundary of $C$ are in $F$. Hence $y\in F$, a contradiction. 

		Next we prove \eqref{eq:lem:DclosebdC}. Take any $y\in C$. If $y$ is in ${C'}$, then $\dist (y,{C'}) = 0\leq \varphi(y)$. Assume that $y\notin {C'}$ (and hence $y\notin U$). 
		{There exists a point $x\in \partial C$ such that $y\in [x,c]$, i.e., 
			\[
			y = x + \lambda \frac{c-x}{\norm{c-x}},
			\]
			for some $\lambda \in [0,\norm{c-x}]$. Since $y$ does not belong to 
			$C'$, we must have $\lambda \in [0,\lambda(x))$ in view of \eqref{eq:ucx} and the definition of $U$.}
		Then 
		\[
		\dist(y,{C'})\leq \dist (y,U) \leq \| y - u(x)\| = |\lambda - \lambda(x)| = \lambda(x) - \lambda \leq \lambda(x) \leq \min_{u\in [x,c]} \varphi(u) \leq \varphi(y),
		\]
		which is what we were after. 
		
		Finally, we show that $\relint C'\subset \relint C$. By our construction we have $c \in \relint C' \cap \relint C \neq \emptyset$. Since $C'$ is a subset of $C$, $\relint C' \cap \relint C \neq \emptyset$ implies $\relint C' \subset \relint C$, which follows from, say, \cite[Theorem~18.2]{Ro97}.	
	\end{proof}
	Regarding Lemma~\ref{lem:DclosebdC},
	we  note that since $\relint C' \subseteq \relint C$ holds and $C$ is compact, we have indeed $C' \subseteq C$. Also, since $F \subseteq C'$ holds it must be a face of $C'$.

	\subsection{Sandwiching convex sets}\label{sec:sandwich}
	
	The goal of this subsection is to prove that we can sandwich a compact convex set with desired facial structure between two nested convex sets that share at most one proper exposed face.
	%that is, to prove the following Lemma~\ref{lem:refinement}.
	
	\begin{lemma}\label{lem:refinement} Let $C \subset \R^n$ be a compact convex set, and suppose that $F \face C$ is a proper nonempty exposed face. Let $C'\subset C$ be a compact convex set such that ${C'}\cap \partial C = F$.
		
		Then there exists a compact convex set $E\subset \R^n$ such that
		\begin{enumerate}
			\item ${C'}\subseteq E\subseteq C$,
			\item $F \subseteq \partial C\cap E$,  and in particular $F$ is a face of $E$, 
			\item every face of $E$ is either a subface of $F$ or an extreme point. 
		\end{enumerate}
	\end{lemma}

	The proof of Lemma~\ref{lem:refinement} depends on two intermediate sandwiching results, 
	Propositions~\ref{prop:approxwtg} and~\ref{prop:sandwich}.
	We begin with Proposition~\ref{prop:approxwtg}, which follows from Lemma~\ref{lem:DclosebdC} with an appropriate choice of $\varphi$. We then establish several technical statements leading up to the proof of Proposition~\ref{prop:sandwich}.

	\begin{proposition}\label{prop:approxwtg} Suppose that $C\subset \R^n$ is a compact convex set, $F\face C$ is a proper exposed face, and $D \subseteq C$ is a compact convex subset  such that $D\cap \partial C = F$. 
		Then there exists a compact convex set $G$ such that $D\subseteq G\subset C$, $\relint G\subseteq \relint C$, $G\cap \partial C = F$  and also $G$ and $C$ share the set of supporting hyperplanes through every $x\in F$.
	\end{proposition}
	The idea here is that we can ``fatten'' $D$ (which may, say, be contained entirely in $\partial C$) to obtain an intermediary convex set $G$ whose relative interior is contained in $\relint C$ and has the same supporting hyperplanes with $C$ at points of $F$.
	\begin{proof}
		Let $H$ be a hyperplane exposing $F$, and for every $x\in C$ let
		\[
		\varphi(x) \coloneqq \dist^2 (x, H).
		\]
		Notice that $\varphi$ is continuous, nonnegative and, for $x\in C$, we have $\dist (x,H) =0$ if and only if $x\in F$. 
		Then by Lemma~\ref{lem:DclosebdC} there exists a compact convex set $Q$ such that $\partial C\cap Q = F$, $\relint Q\subseteq \relint C$ and 
		\begin{equation}\label{eq:xq}
		\dist (x,Q)\leq \varphi(x) = \dist^2(x,H), \quad \forall \, x\in C.
		\end{equation}
		
		Let $G\coloneqq \conv (D\cup Q)$. Observe that the set $G$ is compact, convex and that $G\subseteq C$ holds.
		Let us verify that $G\cap \partial C = F$. 
		Since $F \subseteq D$ and $F \subseteq Q$, we have $G\cap \partial C \supseteq F$. Conversely if $x \in G\cap \partial C$, then $x$ is a strict convex combinations of elements in $D$ and $Q$. 
		Those elements must all be in $\partial C$, since if at least one of them is $\relint C$ we would have $x \in \relint C$. Because $D \cap \partial C = Q \cap \partial C = F$, we conclude that these elements are all in $F$, so $x$ is in $F$ as well.
		
		Since $F$ is a proper face of $C$,
	 the equality $G\cap \partial C = F$ implies that $G$ is properly contained in $C$ as $F$ is properly contained in $\partial C$. 
		
		It remains to verify that $G$ and $C$ have the same set of supporting hyperplanes at the points of $F$. Since $G$ is a subset of $C$, every supporting hyperplane of $C$ at $x \in F$ is a supporting hyperplane of $G$ at $x$. Assume that the converse is not true: there is some point $x\in F$, a normal $p\in \R^n$ with $\norm{p} = 1$ and some $u\in C\setminus G$ such that
		\[
		\langle v-x,p\rangle \leq 0 \quad \forall v\in G, \quad \langle u-x,p\rangle = \alpha >0.
		\]
		Let $\lambda \in [0,1]$ and define
		\[
		u_\lambda \coloneqq x + \lambda (u-x).
		\]
		By convexity, $u_\lambda \in C$. %for all $\lambda\in [0,1]$. 
		Denoting 
		\[
		S \coloneqq \{v\, |\, \langle v-x,p\rangle \leq 0\}, 
		\]
		we observe that $G\subseteq S$ and so 
		\[
		\dist(u_\lambda, G) \geq \dist(u_\lambda, S).
		\]
		On the other hand, $u_\lambda \notin S$, and hence its projection onto $S$ is 
		\[
		w_\lambda =  u_\lambda - \langle u_\lambda-x,p\rangle p.
		\]
		We hence have
		\begin{equation}\label{eq:ug}
		\dist (u_\lambda, G) \geq \|u_\lambda - w_\lambda\| = | \langle u_\lambda-x,p\rangle| = \lambda |\langle u-x,p\rangle| = \lambda \alpha.
		\end{equation}
		
		On the other hand, since $u\notin F$ (recall that $u\in C\setminus G$ and $F\subset G\subset  C$), we must have $u\notin H$ (recall that $H$ is a hyperplane exposing $F$)  and hence $d \coloneqq \dist(u,H)>0$. 
		Let $q$ be the projection of $u$ onto $H$. 
		For any $\lambda \in \R$, letting $q_\lambda \coloneqq x + \lambda (q-x)$, we have $q_\lambda \in H$. We conclude that 
		\[
		\dist(u_\lambda, H) \leq \|u_\lambda - q_\lambda\| = \lambda \|u - q\|. 
		\]

		This together with \eqref{eq:ug}, $G = \conv(D\cup Q)$ and \eqref{eq:xq} leads to
		\[
		\lambda \alpha \leq \dist(u_\lambda,G)\leq \dist(u_\lambda,Q) \leq (\dist (u_\lambda,H))^2 \leq \lambda^2 \|u-q\|^2,  
		\]
		and so
		\[
		\lambda \geq \frac{\alpha}{\|u-q\|^2},
		\]
		which cannot hold for \emph{every} $\lambda \in [0,1]$.  Hence the pair of sets $G$ and $C$ share the set of supporting hyperplanes at all points of $F$.
	\end{proof}

Next, we take care of several technical statements that will be necessary to prove Proposition~\ref{prop:sandwich}. This will be followed by 
a proof of Lemma~\ref{lem:refinement} and then, finally, a proof of Theorem~\ref{thm:amalgamation}.
In what follows, a closed \emph{Euclidean ball} is a set of the form 
$\{u \in \R^n \mid \norm{u-c} \leq r\}$ for some fixed $c \in \R^n$.

	\begin{proposition}\label{prop:ballsep} Let $D\subset \R^n$ be a compact convex set, and suppose that $x\notin D$. Then there exists a closed Euclidean ball $B$ of minimal radius such that this ball separates $x$ from $D$, that is $D\subseteq B$ and  
		$x\in \cl (\R^n\setminus B)$. 
	\end{proposition}
	\begin{proof} 
		Suppose that $D$ is a compact convex set and $x\notin D$. 
		Then, $x$ can be strongly separated from $D$ (e.g., \cite[Theorem~11.4]{Ro97}) and there exists $p\in \R^n$, $\|p\| =1$, such that 
		\[
		d \coloneqq \inf_{u\in D}\langle u,p\rangle - \langle x, p\rangle >0.
		\]
		Since $D$ is compact, there exists some $M>0$ such that 
		\[
		M\geq \sup_{u\in D} \|x-u\|.
		\]
		Let $r \coloneqq \frac{M^2}{2 d}$, $c \coloneqq x + r p$. Then 
		\[
		\|x- c\|  = r,
		\]
		and for any $u\in D$
		\[
		\|u - c\|^2 = \left\| (u-x) - r p \right\|^2 = \|u-x\|^2 - 2 r \langle u-x, p\rangle + r^2  \leq M^2 - 2 d r  + r^2 = r^2,
		\]
		hence $\|u-c\|\leq r$ for all $u\in D$. We have found a ball $B$ centred at $c$ and of radius $r$ containing $D$ and such that $x \in \cl (\R^n\setminus B)$.
		
		It remains to show that there exists a ball of this kind that has  minimal radius. 
		Suppose that $\bar r$ is the infimum over all radii of balls that separate $x$ from $D$; since there is at least one such ball, we must have that $\bar r$ finite and $\bar r\geq 0$. Then there exists some sequence $(B_k)$ of balls with centres $c_k$ and radii $r_k$, such that $r_k\to \bar r$.
		
		Since $(r_k)$ is bounded, $c_k$ is bounded as well. %It is evident that both $c_k$ and $r_k$ are bounded, 
		By selecting a convergent subsequence if necessary, we have $c_k\to \bar c$. 
		Since for every $u\in D$ we have $\|c_k-u\|\leq r_k$, in the limit we obtain $\|\bar c - u\|\leq \bar r$; likewise, since $\|c_k - x\| \geq r_k$ for all $k$, we have $\|\bar c - x\|\geq \bar r$. Therefore the ball of radius $\bar r$ centred at $\bar c$ is a ball of smallest radius separating $x$ from $D$.
	\end{proof}
	
	\begin{proposition}\label{prop:outside} Suppose that $x\in (y,z) \subset \R^n$, where $y\neq z$, and $B$ is a closed Euclidean ball such that $z \in B$ and $x\in \cl(\R^n\setminus B)$. Then $y\notin B$.
	\end{proposition}
	\begin{proof} Suppose that 
		\[
		B = \{u\in \R^n\, |\, \|u-c\|\leq r\}
		\]
		for some $c\in \R^n$ and $r>0$. Then $\|x-c\|\geq r$ and $\|z-c\|\leq r$.
		Assume that on the contrary $y\in B$. %If $x\notin B$, then we get a contradiction with the convexity of B, so this is impossible. 
		Then, the assumption that  $x\in (y,z)$ leads to $x \in B$, 
		which implies that $\|x-c\|=r$. Since $y\in B$, we have 
		\[
		\langle y-c,x-c\rangle \leq \|y-c\|\cdot \|x-c\|\leq r^2.
		\]
		Since $x = (1-t ) z + t y $ for some $t\in (0,1)$, we have 
		\begin{align*}
			\langle y-c, x-c\rangle & =  \langle \frac 1 t x - \frac{1-t}{t} z -c , x-c\rangle \\
			& =  \frac 1 t \langle x-c, x-c\rangle - \frac{1-t}{t} \langle z-c, x-c\rangle \\
			& >  \frac 1 t r^2- \frac{1-t}{t} r^2 = r^2.
		\end{align*}
		Here to obtain the strict inequality we used $\langle z-c, x-c\rangle <r^2$ (from Cauchy-Schwarz we have $\langle z-c, x-c\rangle< \|z-c\|\cdot \|x-c\|\leq r^2$, unless the two vectors featuring in the inner product are collinear, but in that case $\langle z-c, x-c\rangle = r^2$ if and only if $x=z$).
		
		The inequality  $\langle y-c, x-c\rangle>r^2$ contradicts the opposite inequality obtained earlier, hence, our assumption is wrong and $y\notin B$. 
	\end{proof}

	\begin{proposition}\label{prop:infty}
		Suppose that $[u,v] \subset \R^n$ is such that $u\neq v$, and $(w_k) \subset \R^n$ is such that $w_k\to w = \frac 1 2 u + \frac 1 2 v$ and $w_k \notin [u,v]$ for all $k$. Suppose that $(B_k)$ is a sequence of Euclidean balls that separate $w_k$ from $[u,v]$, that is, such that $[u,v]$ is inside $B_k$, and $w_k$ is in the closure of the complement of $B_k$.  Then the centers of these balls depart to infinity, that is, if $B_k$ is centred at $c_k$, then $\|c_k\|\to +\infty$.
	\end{proposition}
	\begin{proof}
		For each ball $B_k$, denote its center and radius by $c_k$ and $r_k$, respectively.
		With that, we have  $\|u-c_k\| \leq r_k,\|v-c_k\|\leq r_k$ and $\|w_k-c_k\|\geq r_k$, for every $k$. 
		
		Let $D \coloneqq \|u-v\|$. Then for any $k$
		\[
		D^2 = \|u-v\|^2 = \|u-c_k\|^2 -  2 \langle u-c_k , v-c_k\rangle + \|v-c_k\|^2  \leq 2 (r_k^2 - \langle u-c_k , v-c_k\rangle),
		\]
		and hence 
		\[
		\langle u-c_k , v-c_k\rangle \leq r_k^2 - \frac{D^2}{2}.
		\]
		We hence have 
		\[
			\norm{c_k-w}^2 = \left\|\frac{1}{2}(c_k-u)+\frac{1}{2}(c_k-v)\right\|^2  = \frac{1}{4}\|c_k-u\|^2  + \frac{1}{4}\|c_k-v\|^2 + 
		\frac{1}{2} \langle c_k-u, c_k-v\rangle \leq r_k^2 - \frac{D^2}{4}.
		\]
		By the triangle inequality
		\[
		\|c_k-w\|\geq \|c_k-w_k\| - \|w-w_k\| \geq r_k - \|w-w_k\|,
		\]
		From the last two inequalities we have
		\[
		\sqrt{r_k^2 - \frac{D^2}{4}} \geq r_k-\|w-w_k\|.
		\]
		Since $w_k$ converges to $w$ and $D \neq 0$, the existence of some subsequence of $r_k$ for which $r_{k_i}$ converges to some finite $\bar{r}$ would lead to a contradiction. 
		Therefore, $\lim_{k \to \infty} r_k = \infty$.
		But this also yields $\|c_k-w\|\to +\infty$, and hence 
		\[
		\|c_k\| \geq \|c_k-w\|-\|w\|\to +\infty. 
		\]
	\end{proof}
	
\begin{proposition}\label{prop:ballstoplanes}
	Suppose that $C \subset \R^n$ is compact  and convex, and the sequences $(c_k)$, $(r_k)$ and $(w_k)$ are such that $\|c_k\|\to \infty$, $c_k/\|c_k\|\to p$, $w_k\to \bar w\in C$, and 
	\begin{equation}
		\label{eq:technical09}
		\|c_k-x\|\leq r_k\leq \|c_k-w_k\|, \; \forall x\in C, \forall k.
	%	\|c_k-w_k\|&\geq r_k, \; \forall k.
	\end{equation}
	Then 
	\begin{equation}\label{eq:xgeq}
		\langle x-\bar w,p\rangle \geq 0, \quad \forall x\in C.
	\end{equation}
	Moreover, if a sequence $(y_k)$ is such that $y_k\to y $ and 
	\[
	\|c_k-y_k\|\geq r_k, \quad \forall k,
	\]
	then
	\begin{equation}\label{eq:ykleq}
		\langle y-\bar w,p\rangle \leq 0.
	\end{equation}
\end{proposition}
\begin{proof} 
	
	From \eqref{eq:technical09} we have  
	\[
	\|c_k - x\|^2\leq r^2_k \leq \|c_k - w_k\|^2,  \; \forall x\in C, \forall k.
	\]
	Expanding and subtracting the first inequality from the second one we obtain 
	\[
	2 \langle c_k, x - w_k\rangle +\|w_k\|^2- \|x\|^2 \geq 0, \quad \forall x\in C.
	\]
	Rearranging and dividing by $\|c_k\|$ (since $\|c_k\|\to \infty$, we can assume that $\|c_k\|\neq 0$ for sufficiently large $k$), we have 
	\[
	\left\langle\frac{ c_k}{\|c_k\|}, x - w_k\right \rangle \geq \frac{\|x\|^2- \|w_k\|^2}{2\|c_k\|},\qquad \forall x\in C.
	\]
	Taking the limit on both sides as $k\to \infty$ (for a fixed $x$) we obtain $\langle x-\bar w,p\rangle \geq 0$ for all $x\in C$. 
	
	Now assume that $(y_k)$ is such that $y_k\to y$ and $\|c_k-y_k\|^2\geq r_k^2$, $\forall k$.	Since $\bar w\in C$, by the first relation of \eqref{eq:technical09} we have 
	$\|c_k-\bar w\|^2\leq r_k^2$, $\forall k$.
	Using the same approach as before, we have 
	\[
	\left\langle\frac{ c_k}{\|c_k\|}, \bar w - y_k\right \rangle \geq \frac{\|\bar w\|^2- \|y_k\|^2}{2\|c_k\|}, \quad \forall k,
	\]
	and taking the limit on both sides as $k\to \infty$ we obtain \eqref{eq:ykleq}.
\end{proof}	
	
	\begin{proposition}\label{prop:sandwich}
		Let $C,D\subset\R^n$ be compact convex sets such that
		\begin{enumerate}%[$(i)$]
			\item $\relint D\subseteq \relint C$; 
			\item for $x\in D\cap \partial C$ the set of supporting half-spaces at $x$ is the same for $D$ and $C$.
		\end{enumerate}
		Then there exists a compact convex set $E$ such that $D\subseteq E \subseteq C$ and the only proper faces of $E$ that are not extreme points lie within $D\cap \partial C$. 
	\end{proposition}
	\begin{proof} We will restrict the space we are working with to the affine hull of $C$, as this will allow us to streamline the proof and avoid unnecessary subtleties related to the construction of properly separating hyperplanes in a higher dimensional space, and such other matters. 
	The affine hull of $C$ is so identified with some Euclidean space $X$.
	In this way, no hyperplane can contain $C$ and the relative interior of $C$ is contained in the interior of every one of its supporting half-spaces.

		We construct $E$ explicitly as follows. For any $x\in D\cap \partial C$ let $H_x$ be some supporting half-space to $C$ at $x$, such that the relative interior of $C$ lies in the interior of $H_x$. 
		For $x\in \partial C\setminus D$ we let $H_x$ be a Euclidean ball of the smallest possible radius that separates $x$ from $D$, which exists by Proposition~\ref{prop:ballsep}. 
		We let
		\begin{equation}\label{eq:defE}
		E \coloneqq \aff C \cap \left(\bigcap_{x\in \partial C} H_x\right).
		\end{equation}
		If $x\in D\cap \partial C$, then $D\subseteq H_x$  holds by assumption $2$.
		If $x\in \partial C\setminus D$, then $H_x$ is a ball containing $D$. 
		Therefore, $D\subseteq E$
%		We have $D\subseteq E$, since $D\subseteq H_x$ for every  $x\in D\cap \partial C$ by assumption~$2$ and also for every $x\in \partial C\setminus D$ .
		
		To see that $E\subseteq C$, assume the contrary and let $y\in E\setminus C$. Take $z\in \relint D\subseteq \relint C$. The line segment connecting $y$ and $z$ must contain a point $x$ on the boundary of $C$, with $x\in (y,z)$. 
		
		Then, there are two cases for the set $H_x$.
		For $x \in D \cap \partial C$, $H_x$ is a supporting half-space of $C$ at $x$ and it contains $z$ in its interior. The segment 
		$(y,z)$ intersects the hyperplane corresponding to $H_x$ at $x$. Now, a line segment intersecting a hyperplane either intersects it in a single point or the line segment is entirely contained in it. 
		We have that $x$ is already in the hyperplane underlying $H_x$ since $H_x$ supports $C$ at $x$.
		We note that $z$ is in the interior of the half-space (i.e., $\relint H_x$), so $y$ cannot be in the same interior, since otherwise the whole 
		line segment $(y,z)$ would be in $\relint H_x$. 
		In addition, $y$ cannot be in the underlying hyperplane either, since otherwise $(y,z)$ would be entirely in it which cannot happen because $z$ is in the interior. 
		The conclusion is that $y$ does not belong to  $H_x$, so $y \not \in E$, which is a contradiction. This takes care of the first case.
		
		The second case is when $x \in \partial C \setminus D$ and $H_x$ is an Euclidean ball. 
		In this case,  we have $z \in \relint D \subseteq H_x$, so	$y\notin E$ by Proposition~\ref{prop:outside}. We conclude that our assumption was wrong and that $E\subseteq C$.
		
		It remains to show that the only proper faces of $E$ outside of $D\cap \partial C$ are extreme points. 
		In order to do so, it is enough to show that if some line segment $[u,v]$ with $u\neq v$ lies within the boundary of $E$, then it must  be a subset of $\partial C$, because this automatically implies that this lime segment is a subset of $D$ as we will check shortly. 
		Indeed, assume the contrary: if there exists a line segment $[u,v]$ that is contained in  $E\cap \partial C$, but is not contained in $D$, since $D$ is closed, there must be a nontrivial subsegment $[\tilde{u},\tilde{v}]$ of $[u,v]$ that is contained in $\partial C\setminus D$ and $\tilde u \neq \tilde v$ holds. 
		The midpoint $\tilde w \coloneqq \frac{1}{2}\tilde u + \frac{1}{2}\tilde v$ of this subsegment is in $\partial C\setminus D$, so the minimum radius ball $H_{\tilde w}$ that separates $\tilde w$ from $D$ contains $E$ by \eqref{eq:defE}.
		Since $\tilde u, \tilde v$ are in $E$, we have $\tilde u, \tilde v \in H_{\tilde w}$ %, so, in fact, 
		%$\tilde w \in H_{\tilde w}$ (by convexity) 
		and $\tilde w$ is in the closure of the complement of $H_{\tilde w}$ (because $H_{\tilde w}$ is a separating ball).
	%	This leads to the conclusion that $\tilde w$ is in the boundary of $H_{\tilde w}$.
		%In particular,  $\tilde w $ is in the closure of the complement of 
		%$H_{\tilde w}$. 
		As $\tilde w \in (\tilde u, \tilde v)$, we 
		conclude that $\tilde u \not \in H_{\tilde w}$ by Proposition~\ref{prop:outside},  which contradicts the fact that $\tilde u$ is in $E \subseteq H_{\tilde w}$.

%		 but then it is impossible for both endpoints of this subsegment to be inside a ball of the smallest radius that separates this midpoint from $D$, and we obtain the contradiction with the fact that this subsegment lies in $E$.
	%	We will now check that this implies that $[u,v]$ is contained in $D$.
%		{\color{red} This automatically implies that $[u,v]$ is contained in $D$, because if $u,v \in \partial C \setminus D$, then $H_u$ and $H_v$ are balls that separate $D$ from $u$ and $v$, respectively, and they appear in the definition of $E$ in \eqref{eq:defE}, so we would have the contradiction that $u \not \in H_u$ and $v \not \in H_v$, in spite of $u$ and $v$ being elements of $E$.} \todo{B: I added the part red but it is either incomplete or wrong. I think we need to use the fact that $H_u$ is minimum here}
		
		Suppose that $[u,v]$ is indeed a segment entirely contained in the boundary of $E$. We will show that it is contained in $\partial C$ and, therefore, in $D$ too. 
		The midpoint $w\coloneqq \frac 1 2 u + \frac 1 2 v$ also lies on the boundary of $E$. In particular, there must be a sequence $(w_k) \subset \aff C$ outside of $E$  converging to $w$. 
		By the construction of $E$ this in turn means that there is a sequence $(x_k) \subseteq \partial C$ such that $w_k\notin H_k \coloneqq H_{x_k}$. Since $C$ is compact, we can assume that $x_k\to x \in \partial C$. 
		Moreover, taking a subsequence if necessary, we can assume that $H_k$ is either a sequence of Euclidean balls or a sequence of half-spaces. 
		We will discuss each case separately and 
		show that either way we have
		\begin{enumerate}[$(a)$]
			\item $x \in D \cap \partial C$;
			\item there exists a half-space $H$ that supports $C$ at $x$  and $E$ at $w$.
		\end{enumerate}
		Before we prove $(a)$ and $(b)$, let us check that they allow us to complete the proof.
		For the sake of obtaining a contradiction, suppose 
		that there exists $\bar w \in [u,v] \cap \relint C$ and denote the hyperplane underlying $H$ by $H^0$.
		Then, \eqref{eq:risegm} implies $\relint\, [u,v] = (u,v)$ is entirely contained in $\relint C$.
		In particular, $w \in \relint C$ so $w$ must be in the interior of $H$, which contradicts the fact that $H$ supports $E$ at $w$ (i.e., $w \in H^0$).
		Therefore, $[u,v]$ must be entirely contained in $\partial C$, which concludes the proof.
%		If there exists some 
%		Since $H$ supports $E$, both $u$ and $v$ are in $H$0
	
		Finally, let us check that $(a)$ and $(b)$ indeed hold.
		First, we consider the case where $H_{k}$ is a sequence of half-spaces supporting $C$ at $x_k$, i.e., we have $H_{k} = \{z \mid \inProd{z-x_k}{p_k} \geq 0 \}$, for a sequence $(p_k)$ of normals.
		By definition, $x_k \in D \cap \partial C$  holds, so $x \in D \cap \partial C$, which proves $(a)$.
		The  normals $p_k$ can be assumed to be on the unit sphere, and hence compact and possessing a limit point $p$. As $x_k \to x$ and $w_k \not \in H_{k}$, $H \coloneqq  \{z \mid \inProd{z-x}{p} \geq 0 \}$ is a supporting hyperplane of $C$ at $x$ satisfying $\inProd{w-x}{p} \leq 0 $. However, 
		$w \in E \subseteq C$, so, in fact, 
		$w$ is in the boundary of $H$.
		Therefore, $H$ is a supporting hyperplane of $E$ at $w$ too, which proves item $(b)$.
%		In addition, since $H_{x_k}$ is a hyperplane, by definition, we have $x_k \in D \cap \partial C$, so $x \in D \cap \partial C$ as well.
		
%		 moreover these hyperplanes intersect some line segments connecting $x_k$ with a fixed relative interior point, and the sequence of these points is bounded, hence has a cluster point. 
		
		 Next, we consider the case where we have a sequence $H_k$ of balls with center $c_k$ and radius $r_k$. We have $w_k \not \in H_{k}$ and 
		 $[u,v] \subseteq H_{k}$ by \eqref{eq:defE}, since $[u,v] \subseteq E$.
		 Therefore, the centers depart to infinity by Proposition~\ref{prop:infty}. Furthermore, since $D$ is compact and these balls must all contain $D$, we have
		 \begin{equation}\label{eq:rk}
		 r_k \to \infty
		 \end{equation}
		 as well.		  
		  Without loss of generality we can assume that $c_k/\|c_k\|$ converges to some unit vector $p$. Moreover, the sequence $(w_k)$ is such that $w_k\notin H_k$ and $w_k\to w \in E$.
		Invoking Proposition~\ref{prop:ballstoplanes} with $(c_k)$, $(r_k)$, $(w_k)$ and $E$, we deduce that the half-space 
		\begin{equation}\label{eq:H}
		H \coloneqq \{z \mid \langle z-w,p\rangle \geq 0\}
		\end{equation}
		contains $E$ and, therefore, contains $D$ as well. Also, since $x_k$ lies in the closure of the complement of the ball $H_k$ for each $k$, applying the second part of Proposition~\ref{prop:ballstoplanes} to $y_k=x_k$ we also deduce that
		\begin{equation}\label{eq:xw}
		\inProd{x-w}{p}\leq 0.
		\end{equation}
		Next, we will verify that  $x\in D \cap \partial C$.
		 Assume the contrary, i.e., $x\in \partial C\setminus D$. Let $q\in \relint D$. Since $D$ is closed, there exists some $\hat q\in (x,q)$ such that $\hat q \notin D$. 
		There exists a Euclidean ball $B$ separating this point $\hat q$ from $D$ by Proposition~\ref{prop:ballsep}.		 
		Since $q \in B$, we conclude that $x$ lies outside of that ball by Proposition~\ref{prop:outside}.
		In particular, there is a small neighbourhood $V$ of $x$ that is also outside of $B$.
		For all $v \in V$, the same $B$ is a closed Euclidean ball that separates $D$ and $v$, so 
		the radius of the smallest Euclidean ball
		separating $v$ from $D$ is upper bounded by the radius of $B$. 
		Hence, the radius of the supporting balls $H_{x_k}$
		cannot go to infinity, which contradicts \eqref{eq:rk}. 
		Our assumption was therefore wrong and we have $x\in D\cap \partial C $, which shows item $(a)$.
		
		Knowing that $x \in D \subset H$, \eqref{eq:xw} implies that $H$ supports $D$ at $x$.
		By assumption~2., we conclude that $H$ also supports $C$ at $x$.		
		%In particular, $H$ is a supporting hyperplane of $D$ at $x$ and, by assumption, must be  a supporting hyperplane of $C$ at $x$ as well.
		We have shown earlier that $E\subseteq C$, so we must have $E\subset H$ too. Since $w\in \partial E$, we conclude that $H$ is a supporting hyperplane to $E$ at $w$ and shows that item~$(b)$ holds. 
	\end{proof}
	
	\begin{proof}[Proof of Lemma~\ref{lem:refinement}] Let $F$, $C$ and ${C'}$  be as in the statement, and suppose that $H$ is some hyperplane that exposes $F$ (as a face of $C$). By Proposition~\ref{prop:approxwtg} there exists a compact convex set $G$ such that ${C'}\subseteq G \subset C$ and  $F= \partial C\cap G$, and such that $G$ and $C$ share supporting hyperplanes at all points of $F$. 
	We can therefore apply Proposition~\ref{prop:sandwich} to establish the existence of a compact convex set $E$ such that ${C'}\subseteq G\subseteq E\subseteq C$,  $F = \partial C\cap G \subseteq \partial C\cap E $, and such that the only faces of $E$ are extreme points and subfaces of $F$. Moreover since $F$ is an exposed face of $C$, and $F\subseteq E\subseteq C$, then $F$ must be an exposed face of $E$.
	\end{proof}

	\begin{proof}[Proof of Theorem~\ref{thm:amalgamation}]
		Under the assumptions of the theorem we invoke Lemma~\ref{lem:DclosebdC} to obtain a compact convex set $C'$ such that 		
		\begin{equation}\label{eq:distrel}
		\dist(x,{C'}) \leq \varphi(x), \; \forall\, x\in C,  
		\end{equation}
		$\relint {C'}\subseteq \relint C$ and $\partial C\cap {C'} = F$. The set $C$, the face $F$ and the set $C'$ satisfy the assumptions of Lemma~\ref{lem:refinement}. Therefore there exists some compact convex set $D$ such that 
		\begin{enumerate}
			\item ${C'}\subseteq D \subseteq C$,
			\item $F\subseteq\partial C\cap D$, and $F$ is a face of $D$, 
			\item every face of $D$ is either a subface of $F$ or an extreme point. 
		\end{enumerate}
	    Since $C'\subseteq D \subseteq C$ from \eqref{eq:distrel} it follows that 
	    \[
		\dist (x,D)\leq \dist(x,{C'}) \leq \varphi(x), \; \forall\, x\in C. 
	    \]
	\end{proof}
	
	\section{Intersections of projectionally exposed cones}\label{sec:counterexample}
	In this section, our main goal is to prove Theorem~\ref{thm:main} and construct two p-exposed cones whose intersection is not p-exposed.
	These cones will be constructed through several steps and we will invoke the results in Sections~\ref{sec:facefix} and \ref{sec:sandwich} along the way.
	
	\subsection{The key construction}\label{sec:construction}
	
	As explained in the introduction, we are going to build a simplified construction to explain how our example `works', before modifying it slightly to obtain the actual sets whose existence proves Theorem~\ref{thm:main}. 
	
	We construct the five-dimensional cones by working on their compact convex four-dimensional slices. 
	For each set, we build its facial structure in a three-dimensional space, then curve this three-dimensional space appropriately to build the boundary of the corresponding four-dimensional compact convex set, by adding a fourth coordinate. 
	We then lift these sets as cones to a five-dimensional space. 
	We use the coordinates $(l,x,y,z,s)$, where $(x,y,z)$ are the three-dimensional coordinates of the three-dimensional face design, $s$ is the `slap on' coordinate that allows us to convert the original three-dimensional structure into the boundary of the four-dimensional set, and the last coordinate $l$ is the lifting coordinate that we use to lift our compact convex set to the conic setting. 
	
	Let $\alpha, \beta, \gamma:[0,1]\to \R^5$ be defined as follows,
	\begin{equation}
		\label{eq:abg}
		\alpha(t) \coloneqq \begin{pmatrix}
			1\\ 0 \\ t^2 \\ t \\ t^3 
		\end{pmatrix}, \;
		\beta(t) \coloneqq \begin{pmatrix}
			1\\ \frac 1 2 t \\ t^3 \\ -\frac{1}{2} t\\ t^3 
		\end{pmatrix}, \; 
		\gamma(t) \coloneqq \begin{pmatrix}
			1\\- \frac 1 2 t \\ t^3 \\ -\frac{1}{2} t\\ t^3  
		\end{pmatrix}.
	\end{equation}
	Let $D$ be the disk
	\begin{equation}\label{eq:defD}
		D \coloneqq \{(1,x,y,0,0) \mid x^2+(y-1)^2 \leq 1\},    
	\end{equation}
	and let %\todo[inline]{Maybe it should be $\stdCone \coloneqq \cone ( D\cup \alpha([0,1])\cup\beta([0,1])\cup\gamma([0,1]))$? The same below    }
	\begin{equation}\label{eq:defK}
		\stdCone \coloneqq \cone ( D\cup \alpha([0,1])\cup\beta([0,1])\cup\gamma([0,1])).    
	\end{equation}
	We will show that 
	\begin{equation}\label{eq:conefd}
	F \coloneqq \cone D
	\end{equation}
	is an exposed face of $\stdCone$ and that it is not projectionally exposed. The following result implies that $F$ is an exposed face of 
	$\stdCone$ (by considering the special case $\hat{D} = D$). 
\begin{lemma}
	\label{lem:exposed-general}
	Let $S\subseteq \R^2$ be a convex set containing the origin. Let $\hat{D} \coloneqq \{(1,x,y,0,0)\in \R^5\,|\, (x,y)\in S\}$. Let $\alpha, \beta, \gamma$ be defined as in~\eqref{eq:abg}.
	Then $\hat{F} \coloneqq \cone \hat{D}$ is an exposed face of 
	$\hat{\stdCone} \coloneqq \cone(\hat{D} \cup \alpha([0,1])\cup \beta([0,1]) \cup\gamma([0,1]))$.
\end{lemma}
\begin{proof}%[{Proof of Lemma~\ref{lem:exposed-general}}]
	Let $u = (0,0,0,0,1)$. We will show that $\hat{F} = \{v\in \hat{\stdCone} \mid \inProd{u}{v} = 0\}$. 
	The inclusion $\subseteq$ holds because $\langle u,v\rangle = 0$ for all $v\in \hat{F}$.
	For the reverse inclusion, we note that $\langle u,\alpha(t)\rangle = \langle u,\beta(t)\rangle = \langle u,\gamma(t)\rangle = t^3$ for all $t\in [0,1]$. Therefore if 
	$v\in \alpha([0,1])$ then $\langle u,v\rangle = 0$ if and only if $v = \alpha(0)\in \hat{D}$.
	Similarly if $v\in \beta([0,1])$ then $\langle u,v\rangle = 0$ if and only if $v = \beta(0)\in \hat{D}$ and if $v\in \gamma([0,1])$ then $\langle u,v\rangle = 0$ if and only if $v = \gamma(0)\in \hat{D}$. By decomposing an arbitrary element $v\in \hat{\stdCone}$ as a conic combination of elements of $\hat{D}$, $\alpha([0,1])$, $\beta([0,1])$ and $\gamma([0,1])$, it follows that $\langle u,v\rangle = 0$ if and only if $v\in \cone\hat{D} = \hat{F}$. 
\end{proof}
% 
% 	Observe that $u = (0,0,0,0,1)$ is an exposing normal for $F$: indeed,  
% 	\[
% 	\langle u, \alpha(t)\rangle = \langle u, \beta(t)\rangle = \langle u, \gamma(t)\rangle = t^3.
% 	\]
% 	Note that $\alpha(0) = \beta(0) = \gamma(0)\in D$, hence,
% 	\[
% 	\langle u, v\rangle = 0 \quad \forall v\in F; \quad \langle u, v\rangle > 0 \quad \forall v\in K\setminus F,
% 	\]
% 	which is easy to see by thinking in terms of conic combinations.

The following result, in the special case $\delta_i = 0$ for $i=1,2,3$,
implies that $F$ is not a projectionally exposed face of $\stdCone$, since it rules out the existence 
of an idempotent linear map $P$  such that $P(\stdCone) = F$. It is also general enough to  
show that $F$ is not a projectionally exposed face of a certain modification of $\stdCone$ that will be 
introduced in Section~\ref{sec:proof}.
	\begin{proposition}
		\label{prop:not-pexposed}
		Let $\alpha,\beta,\gamma$ be defined as in~\eqref{eq:abg}, $D$ as in \eqref{eq:defD} and let $F$ as in \eqref{eq:conefd}.
	%	Let $D \coloneqq \{(1,x,y,0,0) \mid x^2+(y-1)^2\leq 1\}$, \todo{Maybe substitute to a reference to \eqref{eq:defD}?} and let $F \coloneqq \cone(D)$.
		For $i=1,2,3$, let $\delta_i:[0,1]\rightarrow \{0\}\times \R^4$ satisfy $\|\delta_i(t)\|\leq t^3$ for all $t\in [0,1]$. Then there does not exist an idempotent linear map $P:\R^5\rightarrow \R^5$ such that $P(F) = F$, $P(\R^5) \subseteq \lspan F$,  
	$P(\alpha(t) + \delta_1(t)) \subseteq F$ for all $t\in [0,1]$,
		$P(\beta(t) + \delta_2(t)) \subseteq F$ for all $t\in [0,1]$, and
	$P(\gamma(t)+\delta_3(t))\subseteq F$ for all $t\in [0,1]$.
	\end{proposition}
	\begin{proof}%[{Proof of Proposition~\ref{prop:not-pexposed}}]
	We will assume that such a map $P$ exists, and seek a contradiction. 
	We note that $\lspan F = \R^3\times \{(0,0)\}$. Hence if $P:\R^5\rightarrow \R^5$ is idempotent and satisfies $P(F) = F$ then $P$ must fix $\lspan F$. Since $P(\R^5)\subseteq \lspan F$, 
		in coordinates with respect to the standard basis, we must then have that 
	\begin{equation*}%\label{eq:defGenericP2}
		P = \begin{pmatrix}
			1 & 0 & 0 & a_1 & b_1 \\
			0 & 1 & 0 & a_2 & b_2 \\
			0 & 0 & 1 & a_3 & b_3 \\
			0 & 0 & 0 & 0 & 0 \\
			0 & 0 & 0 & 0 & 0
		\end{pmatrix}   
	\end{equation*}
		for some real numbers $a_1,a_2,a_3,b_1,b_2,b_3$. 
%\todo[inline]{Should we comment on $\intercal$?}
	Now 
	\[
	F = \cone D = \{(l,x,y,0,0) \mid x^2+(y-l)^2 \leq l^2,\, l\geq 0\} = \{(l,x,y,0,0) \mid
	2 y l - x^2 - y^2 \geq 0, \,l\geq 0\}.
	\]
	Since $P(\alpha(t)+\delta_1(t))$ and $P(\beta(t)+\delta_2(t))$ and $P(\gamma(t)+\delta_3(t))$ lie in $F$ for all $t\in [0,1]$, it follows that $P(\alpha(t)+\delta_1(t))$, $P(\beta(t)+\delta_2(t))$ and $P(\beta(t)+\delta_3(t))$ must each satisfy the quadratic inequality $2yl-x^2-y^2 \geq 0$. 
		Substituting $P(\alpha(t)+\delta_1(t))$, $P(\beta(t)+\delta_2(t))$ and $P(\gamma(t)+\delta_3(t))$ into $2 y l - x^2 - y^2\geq 0$, grouping terms by degree in $\delta_i(t)$, and grouping the terms that do not involve $\delta_i(t)$ by powers of $t$, gives the inequalities
%\todo[inline]{Got a bit a confused on this part. Why is $p_1$ a polynomial? Opening up the computations}
	\begin{align}\label{eq:ContradictoryInequalities}
		2 a_3 t + (2 - a_2^2 + 2 a_1 a_3 - a_3^2) t^2 +  p_1(t) t^3 + q_1(t)^\intercal\delta_1(t) + \delta_1(t)^\intercal R_1(t) \delta_1(t)& \geq 0,\notag\\
		- a_3 t + \frac{1}{4}(-1 + 2 a_2 - a_2^2 + 2  a_1 a_3 - a_3^2) t^2+  p_2(t) t^3+ q_2(t)^\intercal\delta_2(t)+ \delta_2(t)^\intercal R_2(t) \delta_2(t)& \geq 0,\\
		- a_3 t + \frac 1 4(-1 - 2 a_2 - a_2^2 + 2  a_1 a_3 - a_3^2) t^2+  p_3(t) t^3+ q_3(t)^\intercal\delta_3(t)+ \delta_3(t)^\intercal R_3(t) \delta_3(t)& \geq 0,\notag
	\end{align}
		where $p_1,p_2,p_3$ are polynomials, $q_1,q_2,q_3$ are polynomial maps from
		$[0,1]$ to $\R^5$ and $R_1,R_2,R_3$ are polyomial maps from $[0,1]$ to
		$5\times 5$ symmetric matrices. For each $i=1,2,3$, there exists a positive constant $M_i$ such that  
		\[ q_i(t)^\intercal \delta_i(t) + \delta_i(t)^\intercal R_i(t)\delta_i(t) \leq \|\delta_i(t)\|\sup_{t\in [0,1]}\|q_i(t)\| + \|\delta_i(t)\|^2\sup_{t\in [0,1]}\|R_i(t)\| \leq M_i\,t^3,\]
		where we have used the fact that $\|\delta_i(t)\|\leq t^3$ for all $t\in [0,1]$, and the fact that $\|q_i(t)\|$ and $\|R_i(t)\|$ are continuous.
		Therefore, the
		inequalities~\eqref{eq:ContradictoryInequalities} imply
		\begin{align}\label{eq:ContradictoryInequalities2}
			2 a_3 t + (2 - a_2^2 + 2 a_1 a_3 - a_3^2) t^2 +  t^3(p_1(t) + M_1)& \geq 0,\notag\\
			- a_3 t + \frac{1}{4}(-1 + 2 a_2 - a_2^2 + 2  a_1 a_3 - a_3^2) t^2+  t^3(p_2(t) + M_2)& \geq 0,\\
			- a_3 t + \frac 1 4(-1 - 2 a_2 - a_2^2 + 2  a_1 a_3 - a_3^2) t^2+   t^3(p_3(t)+ M_3)& \geq 0.\notag
	\end{align}
	As $t$ approaches zero, the sign of the lowest order term in $t$ determines the sign of the whole expression. Therefore the first inequality of~\eqref{eq:ContradictoryInequalities2} implies that $a_3 \geq 0$ and the remaining two inequalities of~\eqref{eq:ContradictoryInequalities2} imply that $a_3 \leq 0$. Therefore $a_3=0$. 
	Substituting $a_3=0$ and taking the sum of the last two inequalities gives
	\begin{equation*}
		- \frac{1}{2}(1 + a_2^2)  t^2+ t^3( p_2(t) +p_3(t)  +M_2+M_3)  \geq 0.
	\end{equation*}
		On the one hand, the leading coefficient (i.e., the coefficient of $t^2$) must be nonnegative. On the other hand $-(1+a_2^2)/2 \leq -1/2 < 0$ for all $a_2$, leading to a contradiction.
		Hence there does not exist a projection $P$ with the desired properties.
	\end{proof}

	Enlarging both the face $F$ and the cone $\stdCone$ in a specific way makes the corresponding enlarged face projectionally exposed with respect to the enlarged cone. Let 
	\begin{align}\label{eq:D1}
		D_1 &\coloneqq  \{(1,x,y,0,0) \mid x^3-(y-1)^3 \leq 1, x\geq 0,\; 1\geq y\geq 0\},\\
		\label{eq:D2}
		D_2 &\coloneqq \{(1,x,y,0,0) \mid -x^3-(y-1)^3 \leq 1, x\leq 0,\;1\geq y\geq 0\}.   
	\end{align}
	Let $D$ be the disk defined in~\eqref{eq:defD} and let
	\begin{equation}
		\label{eq:Kidef}
		\stdCone_i \coloneqq \cone ( D \cup D_i \cup \alpha[0,1] \cup \beta[0,1] \cup \gamma[0,1])\quad\textup{for $i\in \{1,2\}$}.
	\end{equation}
	Lemma~\ref{lem:exposed-general}, with the choice $\hat{D} = D_i \cup D$, tells us that
	\begin{equation}\label{eq:def_fi}
	F_i \coloneqq \cone (D_i \cup D)
	\end{equation}
	is an exposed face of $\stdCone_i$, for $i\in\{1,2\}$. The next result tells us that $F_i$ is actually a projectionally exposed face of $\stdCone_i$ for $i\in \{1,2\}$. 
	\begin{proposition}
		\label{prop:KiPe}
		Let $D$ be the disk defined in~\eqref{eq:defD},  let $D_1$ and $D_2$ be defined in~\eqref{eq:D1} and~\eqref{eq:D2} respectively, and let $\stdCone_i$  be defined in~\eqref{eq:Kidef} respectively, for $i\in \{1,2\}$. Then, for $i\in \{1,2\}$, $F_i$ as defined in \eqref{eq:def_fi} is a projectionally exposed face of $\stdCone_i$. 
	\end{proposition}
	\begin{proof}
	Our construction is symmetric with respect to the reflection in the hyperplane orthogonal to the $x$-axis. This transformation preserves $D$ and $\alpha([0,1])$, exchanges $\beta([0,1])$ and $\gamma([0,1])$ and exchanges $D_1$ and $D_2$. Therefore this reflection preserves $\stdCone$ and swaps $\stdCone_1$ and $\stdCone_2$. Hence $F_1$ is a projectionally exposed face of $\stdCone_1$ if and only if $F_2$ is a projectionally exposed face of $\stdCone_2$. So we will only show that the face $F_1$ is a projectionally exposed face of $\stdCone_1$. Let $P_1:\R^5\to\R^5$ be the idempotent linear map with the matrix representation
	\[
	P_1 = \begin{pmatrix}
		1 & 0 & 0 & 0 & 0 \\
		0 & 1 & 0 & -1 & 0 \\
		0 & 0 & 1 & 0 & 0 \\
		0 & 0 & 0 & 0 & 0 \\
		0 & 0 & 0 & 0 & 0
	\end{pmatrix}.
	\]
	Figure~\ref{fig:projections} shows the images of  
	the curves $\alpha, \beta$ and $\delta$ under $P_1$. 
	\begin{figure}[t]
		\centering
		\includegraphics[width=0.45\textwidth]{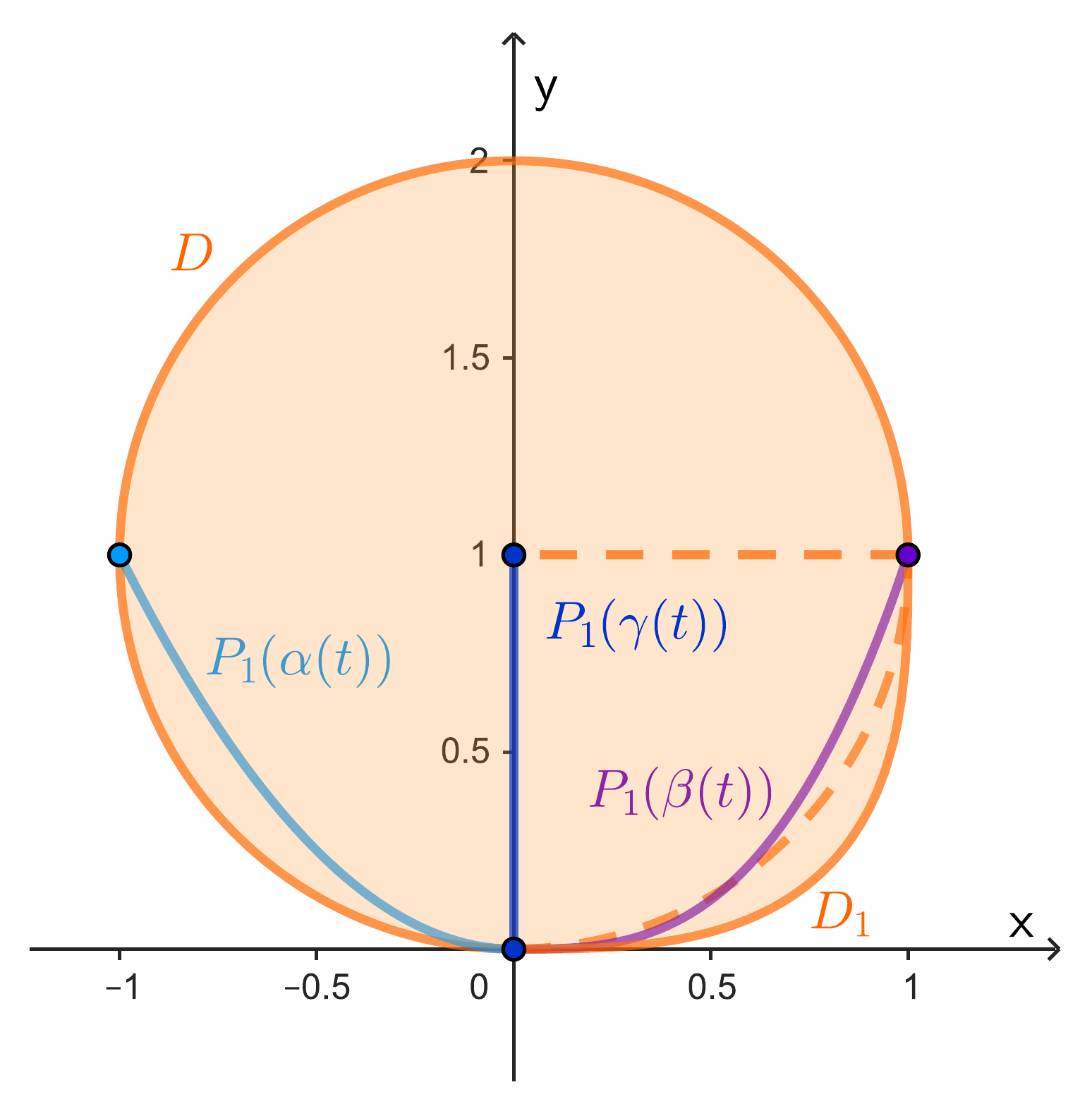}
  		\begin{overpic}[width=0.45\textwidth,%grid
	]{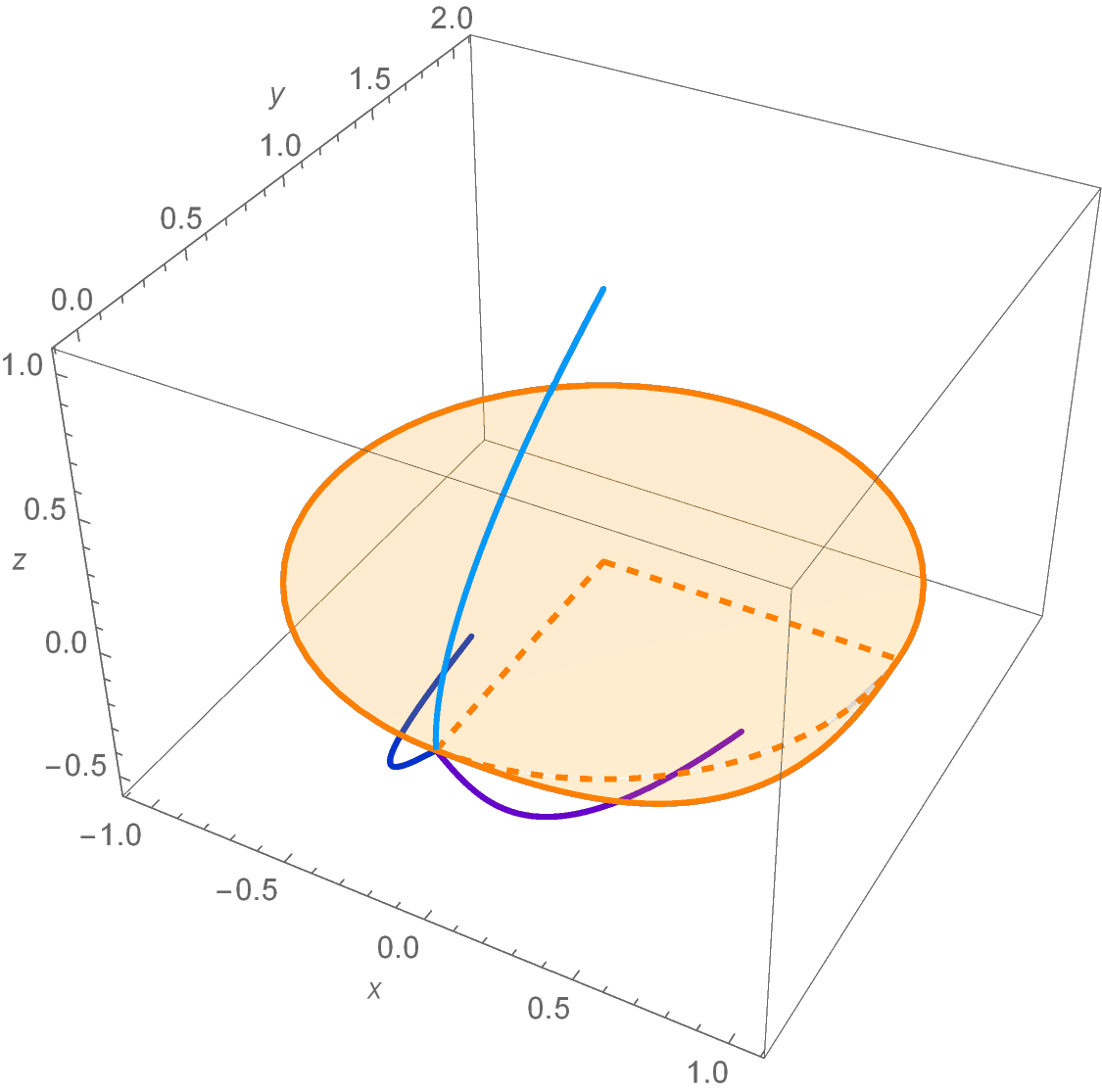} 
	\put(47,70){$\color{lightblue} \alpha$}    
	\put(28,27){$\color{darkblue} \gamma $}    
	\put(52,18){$\color{purple} \beta$}    
\end{overpic}
\caption{Left: the projections of the curves $\alpha,\beta$ and $\gamma$ onto the slice of the face $F_1$ of $\stdCone_1$; right: the $xyz$-projections of these curves, along with the projection of the slice of the face $F_1$ of the cone $\stdCone_1$}
		\label{fig:projections}
	\end{figure}
% 	\begin{align*}
% 	%	P_1 (\alpha(t)) & = (1, -t, t^2, 0, 0),\\
% 		P_1 (\beta(t)) & = (1, t, t^3, 0, 0),\\ 
% 		P_1 (\gamma(t)) & = (1, 0, t^3, 0, 0).
% 	\end{align*}
	Let 
	\[
	(l(t),x(t),y(t),0,0) := P_1(\alpha(t)) = (1, -t, t^2, 0, 0). 
	\]
	Then 
	\[
	2 y(t) l(t) - x^2 (t) - y^2(t) =  2 t^2 - t^2 - t^4 = t^2 (1-t^2)\geq 0,\quad \forall t\in [0,1].  
	\]
	Hence, $P_1(\alpha(t))\in D\subset D \cup D_1$ for all $t\in [0,1]$. %\todo{Should this be $P_1(\alpha(t))\in \cone D$? }
	Likewise, letting 
	\[
	(l(t),x(t),y(t),0,0) := P_1(\gamma(t)) = (1, 0, t^3, 0, 0), 
	\]
	we have 
	\[
	2 y(t) l(t) - x^2 (t) - y^2(t) = 2 t^3 - t^6 = t^3(2-t^3)\geq 0\quad \forall t\in [0,1],  
	\]
	and hence  $P_1(\gamma(t))\in D\subset D_1$ for all $t\in [0,1]$. 
	Finally, for
	\[
	(l(t),x(t),y(t),0,0) := P_1(\beta(t)) = (1, t, t^3, 0, 0), 
	\]
	we have 
	\[
	x^3 (t) - (y(t)-1)^3 -1  =	(1-t) t^3 \left(t^2+t+1\right) \left(t^3-2\right)\leq 0 \quad \forall t \in [0,1],
	\]
	hence we conclude that $P_1(\gamma(t))\in D_1$ for all $t\in [0,1]$. 
	
	Since $P_1$ maps the three curves that span the cone $\stdCone_1$ onto the face $F_1$, it must map the rest of the cone to $F_1$. We conclude that $F_1$ is a projectionally exposed face of $\stdCone_1$.
	\end{proof}
	As discussed previously, to finish the proof of Theorem~\ref{thm:main}, we need to show that the cones $\stdCone_1$ and $\stdCone_2$ are fully projectionally exposed, that is, every face of $\stdCone_1$ and every face of $\stdCone_2$ is projectionally exposed. Instead of proving this result for the cones that we have constructed in this section, we will modify $\stdCone$, $\stdCone_1$ and $\stdCone_2$ slightly to preserve the projectional exposure properties with respect to the faces $F$, $F_1$ and $F_2$, and at the same time modify the boundaries of the cones in such a way that (for $i\in \{1,2\}$) $\stdCone_i$ only has extreme rays and $F_i$ (and its subfaces) as faces. That way the remaining faces of the cones $\stdCone_1$ and $\stdCone_2$ are automatically projectionally exposed.

We observe that the intersection $F_1 \cap F_2$ is $F$ which is geometrically clear from Figure~\ref{fig:projections}, but we will check this formally.
\begin{proposition}\label{prop:f1f2}
	Let $F_i$ be the convex cone defined in \eqref{eq:def_fi}, for $i \in \{1,2\}$. Then, $F_1 \cap F_2 = F$ holds, where $F$ is defined in \eqref{eq:conefd}.
\end{proposition}
\begin{proof}
First we verify that 
\begin{equation}\label{eq:d}
D = \conv(D_1\cup D) \cap \conv(D_2\cup D)
\end{equation}
holds. The inclusion $\subseteq$ is straightforward, so if we focus on the opposite inclusion.
We note that 
\[
\conv(D_i\cup D) = \bigcup _{\lambda \in [0,1]}(\lambda \conv(D_i) + (1-\lambda )D),
\]
for $i \in \{1,2\}$.
For $p = (1,x,y,0,0) \in \conv(D_1\cup D)$, we claim 
that $x \leq 0$ implies that $p \in D$. 
Indeed, suppose that  $p = \lambda q + (1-\lambda)d$ with $q \in \conv(D_1)$, $d \in D$ and $\lambda \in [0,1]$.
If $\lambda = 0$, then  $p \in D$ and we are done. 
Otherwise, $\lambda \in (0,1]$ and the $x$-coordinate of $d$ must be 
nonpositive since the $x$-coordinate of $q$ is always nonnegative, see \eqref{eq:D1}.
In addition, by \eqref{eq:defD} and \eqref{eq:D1}, the $y$-coordinate of $d$ and $q$  are both contained in $[0,2]$.
This means that the line segment connecting $d$ and $q$ passes through a point $\hat q$ in $\{(1,x,y,0,0)\mid x = 0, 0 \leq y \leq 2\}$, which is a set entirely contained in $D$. 
In particular, $p \in D$ since it is a convex combination of $d$ and $\hat q$.

Similarly, if $p = (1,x,y,0,0) \in \conv(D_2\cup D)$, 
then $x \geq 0$ implies $p \in D$.
Both implications together lead to $\conv(D_1\cup D) \cap \conv(D_2\cup D) \subseteq D$, which shows that \eqref{eq:d} indeed holds.

Finally, taking the conic convex hull, we have $\cone D = F \subseteq F_1 \cap F_2$.
Conversely, let $p = (l,x,y,0,0) \in F_1 \cap F_2$, with $l > 0$. 
Then, $p/l = (1,x/l,y/l,0,0) \in \conv(D_1\cup D) \cap \conv(D_2\cup D) = D$. Therefore, $p \in \cone D = F$.
\end{proof}
%	\todo[inline]{I think it may be an easy fact, but we could quickly arguing why extreme rays are always projectionally exposed.  }

	\subsection{Proof of Theorem~\ref{thm:main}}\label{sec:proof}
	Our goal is to construct two closed convex projectionally exposed cones $ \tilde \stdCone_1, \tilde \stdCone_2\subset \R^5$ such that their intersection $ \tilde \stdCone_1\cap  \tilde \stdCone_2$ has a face $F$ that is not projectionally exposed as a face of this intersection.  Therefore $\tilde \stdCone_1\cap \tilde \stdCone_2$  is not projectionally exposed. This proves Theorem~\ref{thm:main}. We achieve this by modifying the earlier construction using the approximation and sandwich results of Lemmas~\ref{lem:DclosebdC} and \ref{lem:refinement}.

	Let $\stdCone$ be defined by \eqref{eq:defK}. Then  
	\begin{align*}
		C & \coloneqq  \{(l,x,y,z,s)\in \stdCone \mid l=1\} \\
		& = \conv ( D\cup \alpha([0,1])\cup\beta([0,1])\cup\gamma([0,1]))
	\end{align*}
	is a compact convex slice of the cone $\stdCone$, and in particular $\stdCone = \cone C$ (here $\alpha$,$\beta$, $\gamma$ are defined by \eqref{eq:abg}, and $D$ is defined by $\eqref{eq:defD}$). 
	The disk $D$ is a proper exposed face of $C$ for the same reason that its conic convex hull $F = \cone C$ is a proper exposed face of $\stdCone$ (the last coordinate of all points in $D$ is zero, and the last coordinate of all points on the three curves is positive, except for the points that belong to $D$). The function 
	\[
	\varphi(l,x,y,z,s) = s
	\]
	satisfies the assumptions of Lemma~\ref{lem:DclosebdC} for the set $C$ and its face $D$. Hence there exists a compact convex set $C'$ such that $C'\subseteq C$, $\partial C\cap C' = D$, and $\dist (x,C') \leq \varphi(x)$ for any $x\in C$.
	
	More specifically, since $\alpha[0,1]$, $\beta[0,1]$ and $\gamma[0,1]$ are in $C$, there exist some $\alpha',\beta',\gamma':[0,1]\to C'$ such that 
	\begin{align*}
		\alpha'(t) & \coloneqq \alpha(t) + \delta^\alpha(t) = \alpha(t) + (0, \delta^\alpha_x(t), \delta^\alpha_y(t), \delta^\alpha_z(t), \delta^\alpha_s(t)),\\
		\beta'(t) & \coloneqq  \beta(t) + \delta^\beta(t) = \beta(t) + (0, \delta^\beta_x(t), \delta^\beta_y(t), \delta^\beta_z(t), \delta^\beta_s(t)),\\
		\gamma'(t) & \coloneqq \gamma(t) + \delta^\gamma(t) = \gamma(t) + (0, \delta^\gamma_x(t), \delta^\gamma_y(t), \delta^\gamma_z(t), \delta^\gamma_s(t)),
	\end{align*}
	and $\|\delta^\alpha(t)\|$, $\|\delta^\beta(t)\|$, $\|\delta^\gamma(t)\|\leq t^3$.
	It follows from Proposition~\ref{prop:not-pexposed} that there is no idempotent linear map $P:\R^5\rightarrow \R^5$ such that $P(F) = F$, $P(\R^5) \subseteq \lspan F$,  
	$P(\alpha'[0,1]) \subseteq F$ for all $t\in [0,1]$,
		$P(\beta'[0,1]) \subseteq F$ for all $t\in [0,1]$, and
	$P(\gamma'[0,1]\subseteq F$ for all $t\in [0,1]$. 
	Therefore, the face $F = \cone D$ is not a projectionally exposed face of the modified cone $\stdCone' \coloneqq \cone C'$.
% 
% 
% 	Proceeding in exactly the same way as we did in Section~\ref{sec:construction} and substituting $P(\alpha'(t))$, $P(\beta'(t))$ and $P(\gamma'(t))$ into the inequality $ 2 y l - x^2- y^2 = 2 y - x^2- y^2 \geq 0$, we observe that we have the same contradiction to the existence of projectional exposure, since the terms of order 0, 1 and 2 in inequalities \eqref{eq:ContradictoryInequalities} remain the same, and the deviation from the original curves is bounded by a cubic term. We conclude that the face $F$ of the modified cone $K'  = \cone C'$ is not projectionally exposed. 

	%\todo[inline]{Check the two subsequent displays and the explanation to see if everything is correct.}
	We now define 
	\[
	C'_1 \coloneqq \conv (C'\cup D_1), \quad C'_2 \coloneqq \conv (C'\cup D_2),
	\]
	and 
	\[
	C_1 \coloneqq \conv (C\cup D_1), \quad C_2 \coloneqq \conv (C\cup D_2),
	\]
	where $D_1$ and $D_2$ are defined by \eqref{eq:D1} and \eqref{eq:D2}. In other words, $C_1$ and $C_2$ are compact bases of the cones $\stdCone_1$ and $\stdCone_2$ defined in~\eqref{eq:Kidef}, and the sets $C_1'$ and $C_2'$ are compact bases of some cones $\stdCone_1'$ and $\stdCone_2'$ that are subsets of $\stdCone_1$ and $\stdCone_2$ that share the faces $F_1 = \cone (D_1\cup D)$ and $F_2 = \cone( D_2 \cup D)$, respectively. 
	
	\begin{figure}
		\centering
		\begin{overpic}[width=0.75\textwidth%,grid
			]{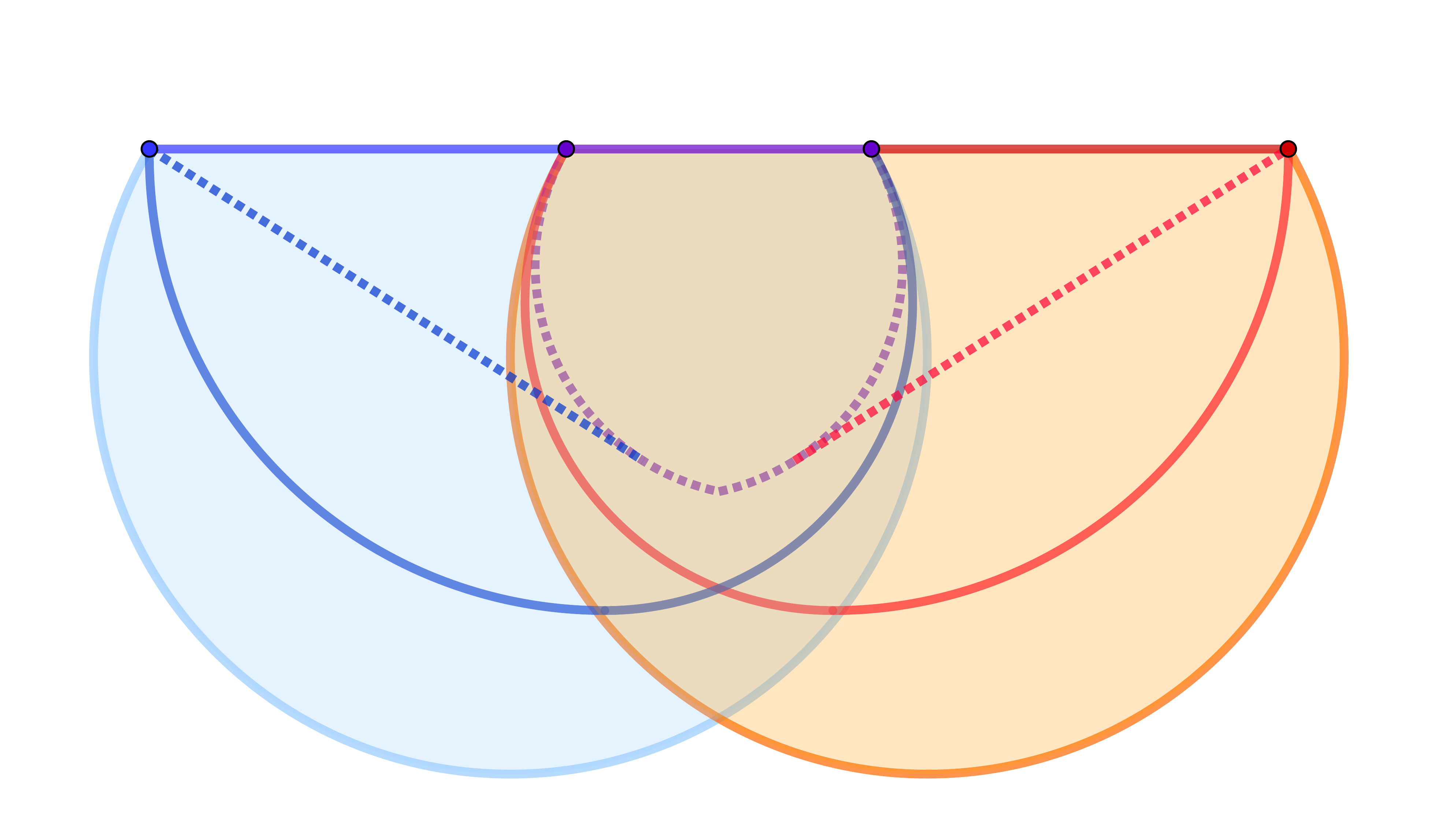} 
			\put(48,50){$F$}    
			\put(30,50){$F_1$}    
			\put(67,50){$F_2$}    
			\put(20,12){$\stdCone_1$}    
			\put(75,12){$\stdCone_2$}    
			\put(47,35){$\stdCone'$}    
			\put(27,41){$\stdCone'_1$}    
			\put(70,41){$\stdCone'_2$}
			\put(25,25){$\tilde \stdCone_1$}    
			\put(70,25){$\tilde \stdCone_2$}   
		\end{overpic}
		\caption{A diagram showing the construction used in proof of Theorem~\ref{thm:main}}
		\label{fig:illproof}
	\end{figure}
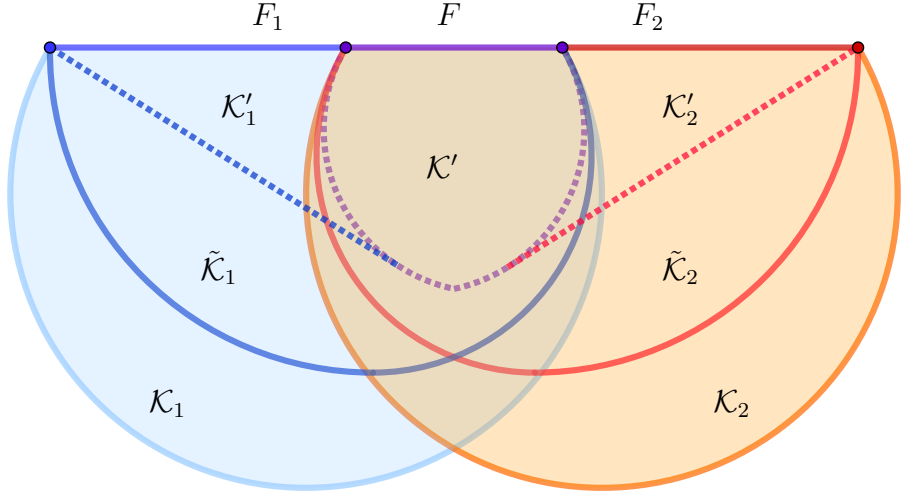
	
	We observe that $C_i$ and $C_i'$ together with the face $\conv (D\cup D_i)$ satisfy the conditions of  Lemma~\ref{lem:refinement} for $i\in \{1,2\}$. 
	Hence we can sandwich some compact convex set  $E_i$ between $C_i'$ and $C_i$ such that the faces of $E_i$ are either extreme points or subfaces of the face $\conv (D\cup D_i)$. 
	Let $\tilde \stdCone_i \coloneqq \cone E_i$ for $i\in \{1,2\}$. Since (for $i\in \{1,2\}$) $F_i$ is still a face of $\tilde \stdCone_i$ and $\tilde{\stdCone}_i \subseteq \stdCone_i$, we conclude from Proposition~\ref{prop:KiPe} that $F_i$ must be a  projectionally exposed face of $\tilde{\stdCone}_i$. 
	Moreover all remaining nonzero faces of these cones are projectionally exposed, since they are all extreme rays~\cite[Proposition 2.2]{BLP87}.
	In addition, by Proposition~\ref{prop:f1f2}, $F_1 \cap F_2 = F$ holds so
	$F$ is a face of  $\tilde \stdCone_1\cap \tilde \stdCone_2$.%, since $F_i$ is a face of $\tilde \stdCone_i$ for $i\in\{1,2\}$.
	
	Overall, we have the following relations for $i \in \{1,2\}$:
	\begin{align*}
	C' \subseteq C \subseteq C_i, & \qquad C' \subseteq C'_i \subseteq E_i \subseteq C_i,\\
	\stdCone' \subseteq \stdCone \subseteq \stdCone_i, & \qquad \stdCone' \subseteq \stdCone_i' \subseteq \tilde \stdCone_i \subseteq \stdCone_i,\\
	F \face \stdCone', & \qquad F_i \face \tilde \stdCone_i
	\end{align*}
	and $\tilde \stdCone_1$, $\tilde \stdCone_2$ are projectionally exposed cones, see also Figure~\ref{fig:illproof}.
		
	It only remains to show that $\tilde \stdCone_1\cap \tilde \stdCone_2$ is not projectionally exposed. 	Assuming the contrary, there exists a projection $P$ that maps $\R^5$ to the linear span of $F$, and such that $P(\tilde \stdCone_1\cap \tilde \stdCone_2)= F$. But then $P(\stdCone')\subseteq P(\tilde \stdCone_1\cap \tilde \stdCone_2)= F$, which means that $P$ projects $\stdCone'$ onto its face $F$, contradicting the fact that $F$ is not a projectionally exposed face of $\stdCone'$. 
	
\section{Conclusions and Some Open Questions}\label{sec:conc}
In this work we developed tools for finding inner approximation of convex sets and this enabled us to prove that there are amenable cones that are not projectionally exposed and that p-exposedness is not necessarily preserved by intersections. 
As usual, answering a question leads to several new ones, so in this final section we take a look at some pending issues.

%In what follows, let $\mathcal{S}^n$ denote the space of symmetric $n\times n$ real matrices and $\PSDcone{n} \subset \mathcal{S}^n$ denote the cone of positive semidefinite matrices in $\mathcal{S}^n$.
%We recall that a closed convex set $C\subseteq \R^m$ is said to be \emph{spectrahedral} if it is linearly isomorphic to $\PSDcone{n}\cap L$ for $L$ some subspace of $\mathcal{S}^n$, see \cite{RG95}.

A first open question is related to the machinery itself.
\begin{quote}
\textbf{Question~1}. Suppose that $C$ belongs to some desirable class of convex sets, e.g., spectrahedral sets or semialgebraic sets.
Is it possible to obtain a version of Theorem~\ref{thm:amalgamation} (or of the other results in Section~\ref{sec:tools}) that ensures that $D$ also belongs to the same class?
\end{quote}
An arbitrary union of semialgebraic sets is not semialgebraic in general, so the $U$ in the proof of Lemma~\ref{lem:DclosebdC} (and, as consequence, the $C'$) may not be semialgebraic even if $C$ is.
This suggests that answering Question~1, if possible at all, may require a new set of techniques.

Next, it would be interesting to explore the class of p-exposed cones.
Spectrahedral cones are known to be amenable and a few of them are known to be projectionally exposed, such as polyhedral cones and symmetric cones. 
This leads to the following natural question.
\begin{quote}
	\textbf{Question~2}. Are spectrahedral cones p-exposed?
\end{quote}
We note that our discussion in Section~\ref{sec:counterexample}
does not answer this question. 
The cones $\stdCone_1$ and $\stdCone_2$ do not appear to be spectrahedral. %, since the Zariski closure of the relative boundaries of their faces $F_1$ and $F_2$ cut through the relative interiors of these faces. \todo{B: Do we need more explanation here?}
However, this does not exclude the possibility that $\stdCone = \stdCone_1 \cap \stdCone_2$ is spectrahedral. %, our example would show  that spectrahedra may not be projectionally exposed. 
Similarly, it would be interesting to explore whether hyperbolicity cones are p-exposed.

Finally, although p-exposedness is not preserved by intersections in general, it may be interesting to look at sufficient conditions for that.
\begin{quote}
	\textbf{Question~3}. Are there reasonable conditions ensuring that the intersection of p-exposed cones $\stdCone_1, \stdCone_2$ is p-exposed? 
\end{quote}
An important case is when one of the cones is a subspace. This is related to Question~2, since spectrahedral cones are positive semidefinite (PSD) cones sliced by a subspace and PSD cones are p-exposed.

In fact, all symmetric cones (in particular, PSD cones) satisfy a stronger property called \emph{orthogonal projectional exposedness} \cite[Proposition~33]{L19}, where each projection onto a face can be taken to be orthogonal.
Related to that, in \cite[Proposition~4.18]{GL23}, there is a discussion on a condition that ensures that a slice of a cone is {orthogonally} projectionally exposed. 
Unfortunately, this condition is far too strong to be generally applicable and it would be helpful to obtain sufficient conditions for p-exposedness under weaker assumptions.

All those three questions are somewhat interconnected. If Question~1 has a positive resolution for, say, spectrahedral cones, this would be useful in constructing a counter-example for Question~2 (if one exists).
Answering Question~3 would also be useful to understand which spectrahedral cones are p-exposed.

\section*{Acknowledgements}
Bruno F. Louren\c{c}o is supported partly by the JSPS Grant-in-Aid for Early-Career Scientists 23K16844.

James Saunderson is the recipient of an Australian Research Council Discovery Early Career Researcher Award 
(project number DE210101056) funded by the Australian Government.

%
%	\begin{remark} Our construction doesn't answer the question whether spectrahedral cones are projectionally exposed. The cones $K_1$ and $K_2$ are not spectrahedral, since the Zariski closure of relative boundaries of their faces $F_1$ and $F_2$ cut through the relative interiors of these faces. However if $K$ is spectrahedral, our example would show  that spectrahedra may not be projectionally exposed. 
%\end{remark}
%
%\begin{remark}
%	Another interesting question is whether we can construct inner approximations similar to the ones presented in the paper in a way that preserve certain classes of sets, e.g., semialgebraic?
%\end{remark}

%
%\todo[inline]{Reasonable conditions for intersections of p-exposed cones to be p-exposed cones. Maybe mention the result from the work with Jo\~ao.}
%
%	\todo[inline]{B: The bibliographic items need some adjustments. Sometimes journal names are abbreviated, sometimes not and etc.}
	
	\bibliographystyle{plain}
	\bibliography{refs}

\end{document}